\pdfoutput=1
\documentclass%[draft]
{amsart}

\usepackage{amsmath, amsthm, amsfonts,amssymb}
\usepackage{mathrsfs}                      % Nice \mathscr fonts
\usepackage{mathtools}                     % \mathrlap
\usepackage{booktabs}
\usepackage{todonotes}
\usepackage[all]{xy}
\usepackage{chngcntr}                      % To use \counterwithin
\theoremstyle{plain}
\newtheorem{theorem}{Theorem}[subsection]

\newtheorem*{proposition*}{Proposition}
\newtheorem{lemma}[theorem]{Lemma}
\newtheorem*{lemma*}{Lemma}
\newtheorem{corollary}[theorem]{Corollary}

\newtheorem*{corollary*}{Corollary}
\newtheorem*{theorem*}{Theorem}
\theoremstyle{definition}
\newtheorem*{remark*}{Remark}

\allowdisplaybreaks

\def\al{\alpha} 
 
\def\ga{\gamma} 
 
\def\ep{\varepsilon}

\def\th{\theta} 
\def\io{\iota} 
\def\ka{\kappa} 
\def\la{\lambda} 
 
\def\si{\sigma} 
\def\ta{\tau} 
\def\ph{\varphi}

\def\om{\omega} 
\def\Ga{\Gamma} 
\def\De{\Delta}

\def\Ph{\Phi} 
 
\def\Om{\Omega}
 
\def\o{\circ} 
\def\i{^{-1}} 
\def\x{\times}
\def\p{\partial}

\def\R{{\mathbb R}}

\def\exp{\operatorname{exp}}

\let\on=\operatorname
\let\wt=\widetilde

\let\ol=\overline
\newcommand{\ud}{\,\mathrm{d}}

\title[$R$-transforms for Sobolev $H^2$-metrics on spaces of plane curves] 
{$R$-transforms for Sobolev $H^2$-metrics on spaces of plane curves}
\author{Martin Bauer, Martins Bruveris, Peter W. Michor}
 
\address{
Martin Bauer, Peter W.\ Michor: Fakult\"at f\"ur Mathematik,
Universit\"at Wien, \newline
Oskar-Morgenstern-Platz 1, A-1090 Wien, Austria.\newline\indent
Martins Bruveris: Institut de math\'ematiques, EPFL, CH-1015 Lausanne, Switzerland}
\email{bauer.martin@univie.ac.at}
\email{martins.bruveris@epfl.ch}
\email{peter.michor@univie.ac.at}

\date{{\today} } 
\dedicatory{Dedicated to David Mumford on the occasion of his 76th birthday} 
\thanks{Martin Bauer was supported by `Fonds zur
F\"orderung der wissenschaftlichen                    
Forschung, Projekt P~24625'} 

\begin{document}
\begin{abstract} 
We consider spaces of smooth immersed plane curves (modulo translations and/or rotations), equipped 
with reparameterization invariant weak Riemannian metrics involving second derivatives. This 
includes the full $H^2$-metric without zero order terms.  We find 
isometries (called $R$-transforms)
from some of these spaces into function spaces with simpler weak Riemannian metrics, and we use 
this to give explicit formulas for geodesics, geodesic distances, and sectional curvatures. 
We also show how to utilise the isometries to compute geodesics numerically.
\end{abstract}
\maketitle

\tableofcontents

\section{Introduction}

%\subsection{Reparameterization invariant metrics on spaces of plane curves}
In this article we will study four different Sobolev $H^2$-type metrics on  the infinite dimensional 
manifold of parametrized curves in the plane, $\on{Imm}(S^1,\R^2)$. This space is of interest due to its connections to the field of mathematical shape analysis. Riemannian metrics are used in shape analysis, since they equip the space with a distance function that can be used for comparison or classification of objects; they also allow to locally linearize the space via the exponential map and thus to generalize linear statistical methods to these -- in general -- highly nonlinear spaces.

In applications to shape analysis one is mostly interested not in the curve 
itself, but only in the shape that it represents. Two curves represent the same shape, if they differ 
by a reparameterization or relabelling of the points.         
For this reason we will only be interested in metrics that are invariant under the action of the 
reparameterization group $\on{Diff}(S^1)$. 

The arguably simplest reparametrization invariant metric on $\on{Imm}(S^1,\R^2)$ is the $L^2$-metric
\[
G_c(h,k) = \int_{S^1} \langle h, k \rangle \ud s\,.
\]
Here $h,k\in T_c\on{Imm}(S^1,\R^2)$ are tangent vectors with foot point $c\in \on{Imm}(S^1,\R^2)$ 
and $\ud s$ 
denotes arc-length integration, i.e., $\ud s = |c'(\theta)| \ud \theta$. Unfortunately this metric is unsuitable for shape analysis, because the induced geodesic distance vanishes, i.e., any two curves can be joined by paths of arbitrary short length. This surprising result was proven first for the quotient space $\on{Imm}(S^1,\R^2)/\on{Diff}(S^1)$ in \cite{Michor2006c}; later it was generalized in \cite{Michor2005} to the space $\on{Imm}(M, N)/\on{Diff}(M)$ of type $M$ submanifolds of $N$ where $M$ is compact and $\on{dim} M \leq \on{dim}N$, as well as the diffeomorphism group $\on{Diff}(M)$; using a combination of both results it was shown in \cite{Bauer2012c} that the distance also vanishes on $\on{Imm}(M,N)$.
% The main drawback of the $L^2$-metric is that the induced geodesic distance vanishes, i.e., the 
% distance between any two curves  is 0, which renders this metric as rather unsuitable for shape 
% analysis.  This highly surprising result was first proven by Mumford and Michor directly on the 
% shape space $B_i(S^1,\R^2)$  in \cite{Michor2006c} and was  generalized to shape spaces of 
% arbitrary dimension and codimension and also to the full diffeomorphism group in \cite{Michor2005}. 
% Vanishing geodesic distance  for the manifold of parametrized curves $\on{Imm}(S^1,\R^2)$ is proven 
% in \cite{Bauer2012c}, and it is a combination of the results on $B_{i,f}(S^1,\R^2)$ and on 
% $\on{Diff}(S^1)$. 
Note that this is a purely infinite dimensional phenomenon -- in finite 
dimensions  the geodesic distance is always positive,  due to  the local invertibility of the 
exponential map.         

The vanishing of the geodesic distance for the $L^2$-metric led to the search for stronger metrics, that would be suitable for shape analysis. Candidates, that have been considered, include the $L^2$-metric weighted by curvature \cite{Michor2006c}:
$$G^A_c(h,k)=\int_{S^1}(1+A\ka_c^2)\langle h,k\rangle\ud s\,,$$
or the length of the curve \cite{YezziMennucci2005,Shah2008}:
$$G^{\Ph}_c(h,k)=\Ph(\ell_c)\int_{S^1}\langle h,k\rangle\ud s\,.$$
Here $\ka_c$ denotes the curvature of the curve, $\ell_c$ its length  and $\Ph:\R\to\R_{>0}$ is a 
suitable positive function.  
These metrics have been generalized to higher dimensional immersions in 
\cite{Bauer2012a,Bauer2012b}.  

A different approach to strengthen the metric and the one, that we will use in this article is to add terms involving higher derivatives of the tangent vectors 
to the metric, leading to metrics of the form
\[
G_c(h,k) = \int_{S^1} \sum_{j=0}^{k} a_j \langle D_s^j h, D_s^j k \rangle \,ds \,,
\]
where $D_sh = \tfrac 1{|c'|} h'$ denotes the arc-length derivative of $h$ and $a_j$ are weights, 
possibly depending on the curve $c$.  
 More generally one can  consider metrics that are defined via a field of symmetric pseudo- differential operators $L_c: 
 T_c\on{Imm}(S^1,\R^2)\to T_c\on{Imm}(S^1,\R^2)$ by 
$$G_c(h,k) = \int_{S^1} \langle L_c h, k \rangle \,ds= \int_{S^1} \langle h, L_c k \rangle \,ds \,.$$
This approach leads to the class of Sobolev-type metrics, which were independently introduced 
in \cite{Charpiat2007, Michor2007, Sundaramoorthi2007} and studied further in 
\cite{Bauer2011b, Bauer2012d, Mennucci2008}.  
Often the operator field $L$ will have a kernel and thus $G^L$ will be a metric only a certain quotient
of $\on{Imm}(S^1,\R^2)$, e.g., if all constant 
vector fields  are in its kernel, then one has to pass to the quotient $\on{Imm}/\on{Tra}$ 
of plane parametrized curves modulo translations. 
An overview of the various metrics on $\on{Imm}(S^1,\R^2)$ can be found in \cite{Bauer2013b_preprint}.

While Sobolev-type metrics are a natural generalization of the $L^2$-metric, their numerical 
treatment is unfortunately rather involved. This stems mainly from the fact that the geodesic equation of a 
Sobolev-type metric of order $k$ is generally a highly nonlinear PDE of order $2k$. There are exceptions.  For the family 
of first order metrics on $\on{Imm}(S^1,\R^2)/\on{Tra}$ given by
\[
G^{a,b}_c (h,h)= \int_{S^1} a^2\langle D_sh,n\rangle^2+b^2 \langle D_sh,v\rangle^2\,ds\,,
\]
with $a, b > 0$ there exists an isometric transformation of the space $\on{Imm}(S^1,\R^2)/\on{Tra}$, called the $R$-transform, which tremendously simplifies the computation of geodesics; see  
\cite{Michor2008a, Sundaramoorthi2011,Klassen2004, Bauer2012b_preprint}.  
% In these articles it has been shown that these metrics can be represented 
% as a pull back of the flat (not reparameterization invariant)  $L^2$-metric on $C^\infty(S^1,\R^d), 
% d\geq2$. This representations allow to gain  results concerning curvature and existence of geodesics 
% and in some special cases one even obtains explicit formulas for geodesics.     
% However, it seems that $H^1$-type metrics are still too weak for many applications in shape 
% analysis, see \cite{Bauer2012b_preprint}.
Apart from simplifying the computations, the representation via the $R$-transform also permits us to compute the curvature and in some special cases to obtain explicit formulas for geodesics.
 
There have been some attempts to solve the geodesic equation directly for order one metrics on curves \cite{Mio2007} and surfaces \cite{Bauer2011a}. Metrics of higher order on the other hand are still practically untouched. The only exception is \cite{Shah2010}, discussing the homogenous $H^2$-metric on the space of plane curves modulo similitudes. 
It is therefore of interest to develop representations of higher order metrics, that have the potential to simplify computations of geodesics.

In this article we continue the investigation started in \cite{Bauer2012b_preprint} and use similar methods to study four different $H^2$-type metrics, namely:
\begin{align*}
G_c(h,k)&=\int_{S^1} \ka^{-3/2}\langle D^2_sh,n\rangle\langle D^2_sk,n\rangle + 
\langle D_sh,v\rangle\langle D_sk,v\rangle\ud s\,,
\tag{\eqref{metric1} in Sect.~\ref{Sec:metric1}}
\\
G_c(h,k) &= \int_{S^1} \langle D_s h, v\rangle\langle D_s k, v\rangle  
+ \langle D_s^2 h, n \rangle \langle D_s^2 k, n \rangle \ud s\,,
\tag{\eqref{metric2} in Sect.~\ref{Sec:metric2}}
\\
G_c(h,k) &= \int_{S^1} \langle D_s h, D_s k \rangle  + \langle D_s^2 h, n \rangle \langle D_s^2 k, n \rangle \ud s\,,
\tag{\eqref{metric3} in Sect.~\ref{Sec:metric3}}
\\
G_c(h,k) &= \int_{S^1} \langle D_s h, D_s k \rangle  + \langle D_s^2 h, D_s^2 k\rangle \ud s\,.
\tag{\eqref{metric4} in Sect.~\ref{Sec:metric4}}
\end{align*}
Despite its seemingly complicated nature, metric \eqref{metric1} is completely amenable to the 
$R$-tranform treatment: For open curves it is flat and we get explicit formulas for geodesics and the geodesic 
distance. The image of the $R$-transform of the space of closed curves is a codimension 2 splitting  
submanifold of an open (with respect to a finer topology) set in a pre-Hilbert space. See Sect.
\ref{Sec:metric1:motivation} for an explanation of the form of this metric. 

For the metric \eqref{metric2} the image of the corresponding $R$-transform for open curves is 
$C^\infty([0,2\pi],(\mathbb R_{>0}\x \mathbb R,g))$ with a weak $L^2$-type metric; 
here $g$ is a curved metric on $\R_{>0} \x \R$, for which we manage to derive (somewhat) explicit formulas 
for geodesics.  The image of the space of closed curves is again a codimension 2 splitting 
submanifold.

For metric \eqref{metric3} the image of the space of open curves under the $R$-transform is 
splitting submanifold of infinite codimension in $C^\infty([0,2\pi], 
(\mathbb R_{>0}\x S^1\x \mathbb R,g))$ described by a system of ODEs. 

The picture for metric \eqref{metric4} is again more complicated but managable; we do not include 
full results in this paper. 

\section{Background material and notation}

In this paper we use convenient analysis in infinite dimensions as described in \cite{Michor1997}.

\subsection{Notation}\label{notation}
Let $M$ denote either $S^1$ or $[0,2\pi]$ and let $c:M \to \R^2$ be a regular curve, i.e., $c'(\th) \neq 0$. We denote the curve parameter by $\th \in M$ and 
differentiation by $'$, i.e., $c'=\p_\th c$. Since $c$ is an immersion, the unit-length 
tangent vector $v = c' / |c'|$ is well-defined. Denote by $J$ the rotation by $\tfrac \pi 2$. 
Rotating $v$ we obtain the unit-length normal vector      
$$n = \left(\begin{array}{cc}
0 & -1  \\
1 & 0
\end{array}\right)v=Jv\,.$$ 
We will denote by $D_s = \tfrac 1{|c'|} \p_\th$ the derivative with respect 
to arc-length and by $\ud s = |c'| \ud \th$ the integration with respect to arclength. To 
summarize, we have  
\begin{align*}
v &= D_s c\,, & n &= Jv\,, & D_s &= \frac 1{|c'|} \p_\th\,, & \ud s &= |c'| \ud \th\,.
\end{align*}
The curvature can be defined as 
\[
\ka = \langle D_s v, n \rangle\,,
\]
where $\langle\,,\,\rangle$ denotes the Euclidean inner product.
The orthonormal frame $(v,n)$ satisfies the Frenet equations,
\begin{align*}
D_s v &= \ka n \\
D_s n &= -\ka v\,.
\end{align*}

We define the {\it turning angle} $\al : M \to \mathbb R/2\pi \mathbb Z$ of a curve $c$ by 
$v(\th) = \exp(i \al(\th))=(\cos \al, \sin \al)$; 
we shall often treat $S^1$, $\mathbb R/2\pi\mathbb Z$, and the intervall $[0,2\pi]$ 
with endpoints identified, as the same space.

\subsection{The manifold of plane curves}\label{Background:plane curves}
The space of \emph{closed immersed curves}, 
\[
\on{Imm}(S^1, \R^2) = \left\{ c \in C^\infty(S^1, \R^2) \,:\, c'(\th) \neq 0 \right\}\,,
\]
is an open set in the manifold $C^\infty(S^1, \R^2)$ with respect to the $C^\infty$-topology and 
thus itself a smooth manifold. The tangent space of $\on{Imm}(S^1, \R^2)$ at the point $c$ consists 
of all vector fields along the curve $c$. It can be described as the space of sections of the 
pullback bundle $c^\ast T\R^2$,   \begin{equation*}
T_c \on{Imm}(S^1,\R^2) = \Ga(c^\ast T\R^2) =
\left\{h: \quad \begin{aligned}\xymatrix{
& T\R^2 \ar[d]^{\pi} \\
S^1 \ar[r]^c \ar[ur]^h & \R^2
} \end{aligned} \right\}\,.
\end{equation*}
Since the tangent bundle $T\R^2$ is trivial, we can identify $T_c \on{Imm}(S^1,\R^2)$ with the space 
of $\R^2$-valued functions on $S^1$, 
\[
T_c \on{Imm}(S^1, \R^2) \cong C^\infty(S^1, \R^2)\,.
\]
If we drop the periodicity condition we obtain the manifold of \emph{open immersed curves},
\[
\on{Imm}([0,2\pi], \R^2) = \left\{ c \in C^\infty([0,2\pi], \R^2) \,:\, c'(\th) \neq 0 \right\}\,.
\]
The tangent space of $\on{Imm}([0,2\pi], \R^2)$ is similar to that for closed curves, only with $S^1$ 
replaced by $[0,2\pi]$. Whenever we describe results that work for both open and closed curves 
we will write $\on{Imm}(M,\R^2)$, with $M$ standing for either $S^1$ or $[0, 2\pi]$.

All metrics in this article will be degenerate on $\on{Imm}(M,\R^2)$ -- their kernel will consist 
of translations or Euclidean motions. 
This leads us to consider the spaces $\on{Imm}(M,\R^2)/\on{Tra}$ and $\on{Imm}(M,\R^2)/\on{Mot}$. 
Here $\on{Tra}$ denotes the translation group on  
$\R^2$, i.e., $\on{Tra}\cong \R^2$ and $\on{Mot}$ denotes the group $\mathbb R^2\ltimes SO(2)$ of 
Euclidean motions. We will identify the quotients with subspaces of $\on{Imm}(M,\R^2)$ (sections 
for the left action of the translation group or the motion group) in the 
following way  
\begin{align*}
 \on{Imm}(M,\R^2)/\on{Tra}&\cong\{c\in \on{Imm}(M,\R^2):c(0)=0\} \\
 \on{Imm}(M,\R^2)/\on{Mot}&\cong\{c\in \on{Imm}(M,\R^2): c(0) = 0 \text{ and } \al(0) = 0\}\,.
\end{align*} 
The tangent spaces are then given by
\begin{align*}
T_c\left( \on{Imm}(M,\R^2)/\on{Tra}\right)&\cong\{h\in C^{\infty}(M,\R^2): h(0)=0\} \\
T_c\left( \on{Imm}(M,\R^2)/\on{Mot}\right)&\cong\{h\in C^{\infty}(M,\R^2): h(0)=0, \langle D_sh(0),n(0)\rangle=0\}\,.
\end{align*}
If we want sections that are invariant under the reparameterization group $\on{Diff}(S^1)$, we shall 
consider instead
\begin{align*}
 \on{Imm}(M,\R^2)/\on{Tra}&\cong\{c\in \on{Imm}(M,\R^2): \int_M c \ud s = 0\}\,,\\
 \on{Imm}(M,\R^2)/\on{Mot}&\cong\{c\in \on{Imm}(M,\R^2): \int_M c \ud s = 0 \text{ and } \int_M \al \ud s = 0\}\,.
\end{align*}
If we mean any of these sections, we shall write $\mathcal C(M,\mathbb R^2)$: in all cases it is 
the intersection of $\on{Imm}(M,\mathbb R^2)$ with a closed linear subspace of 
$C^\infty(M,\mathbb R^2)$. 
We will also need the space of positively oriented convex curves
\begin{align*}
\on{Imm}_{\on{conv}}(M,\R^2):=\{c\in \on{Imm}(M,\R^2): \ka(c) > 0 \}\,, 
\end{align*}
which is an open set in $\on{Imm}(M, \R^2)$ and thus itself a smooth manifold. 

\subsection{Variational formulae}\label{var}
We will need formulae that express, how the quantities that have been introduced in the previous 
sections change, if we vary the underlying curve $c$.  
For a smooth map $F $ from $\on{Imm}(M,\R^2)$ to any convenient vector space  we denote by
\[
dF(c).h = D_{c,h}F = \left.\frac{\ud}{\ud t}\right|_{t=0}F(c + th)  = 
\left.\frac{\ud}{\ud t}\right|_{t=0}F(\wt c(t,\quad))
\]
the variation  in the direction $h$, 
where $\wt c:\mathbb R\x M\to \mathbb R^2$ is any smooth variation with $\wt c(0,\th)= c(\th)$ and 
$\p_t|_0 \wt c(t,\th)=h(\th)$ for all $\th$.
Examples of maps $F$ include $v$, $n$, $\al$, $|c'|$, $\ka$.
In the following lemma we collect the basic variational formulae that we will use throughout the article.
\begin{lemma}
The first variations of the the turning angle $\al$, the unit tangent vector $v$, the normal vector 
$n$, the length element $|c'|$ and the curvature $\ka$ are given by  
\begin{align}
\label{var_al}
d\al(c).h &= \langle D_s h, n \rangle \\
\label{var_v}
dv(c).h &= \langle D_s h, n \rangle n \\
\label{var_n}
dn(c).h &= -\langle D_s h, n \rangle v \\
d(|c'|).h &= \langle D_s h, v \rangle |c'| 
\label{var_length}\\
d\ka(c).h &= \langle D_s^2 h, n \rangle - 2\ka \langle D_s h, v \rangle\,.
\label{var_ka}
\end{align}
\end{lemma}
\begin{proof}
The proof of these formulae can be found for example in  \cite{Michor2007}. 
\end{proof}

%% 
% These and other identities can be derived using the following identity.
% 
% \begin{lemma} If $R$ is a smooth map $R : \on{Imm}(S^1,\R^2) \to C^\infty(S^1, \R^n)$, then the derivative of the composition $D_s \o R$ is given by
% \[
% T_c \left(D_s \o R\right).h = D_s \left(T_c R.h\right) - \langle D_s h, v \rangle D_s R(c)\,.
% \]
% \end{lemma}
% 
% \begin{proof}
% The operator $\p_\th$ is linear and thus commutes with the derivative with respect to $c$. Thus we have
% \begin{align*}
% T_c \left(D_s \o R\right).h &= T_c \left( \frac1{|c_\th|} \p_\th R(c) \right).h \\
% &= \frac 1{|c_\th|} \p_\th\left(T_c R.h\right) + \left(T_c \frac{1}{|c_\th|}.h\right) \p_\th R(c) \\
% &= D_s \left( T_c R.h \right) - \frac 1{|c_\th|} \langle D_s h, v \rangle \p_\th R(c)\,.
% \end{align*}
% and this concludes the proof.
% \end{proof}

\subsection{Riemannian metrics on spaces of curves}
A Riemannian metric on the manifold of curves is a smooth family of
positive definite inner products $G_c(.,.)$ with 
$c \in \on{Imm}(M,\R^2)$, i.e.,
\[
G_c : T_c \on{Imm}(M,\R^2) \x T_c \on{Imm}(M,\R^2) \to \R\,.
\]
Each metric is weak in the sense that $G_c$, viewed as linear map 
\[
G_c : T_c\on{Imm}(M,\R^2) \to \left(T_c \on{Imm}(M,\R^2)\right)'
\]
from $T_c\on{Imm}(M,\R^2)$ into its dual, which  consists of $\mathbb R^2$-valued distributions on 
$M$, is injective but not surjective.  
In this article, we will study metrics  $G$ that are induced by an operator field $L$ via
\begin{equation*}
G_c(h,k) = \int_M \langle L_c h, k\rangle \ud s=\int_M \langle  h, L_ck\rangle \ud s\,,
\end{equation*}
with $c$ a curve and $h,k\in T_c \on{Imm}(M,\R^2)$. 
The \emph{momentum} $p = L_c h \otimes  \ud s\in (T_c \on{Imm}(M,\R^2))'$ allows us to represent the metric as
\[
G_c(h,k) =  \langle p, k \rangle_{T\!\on{Imm}} \,.
\]
Furthermore, we will be interested only in metrics that are invariant under the action of the 
reparameterization group $\on{Diff}(M)$, that is, for each $\ph \in \on{Diff}(M)$ the metric has to 
satisfy  
\[
G_{c\o\ph}(h\o\ph,k\o\ph) = G_c(h,k)\,.
\]
In this case the operator field $L$ inducing the metric is invariant under  reparameterizations.

\begin{remark*}
The reason for this restriction is that in applications to shape analysis one is mostly interested 
not in the curve $c$ itself, but only in the shape that the curve represents. Two curves $c$ and $e$ 
represent the same shape, if they differ by a reparameterization or relabelling of the points, i.e., 
$c = e \o \ph$. Thus one passes to the quotient   
$$\on{Imm}(M,\R^2)\to B_i(M,\R^2):=\on{Imm}(M,\R^2)/\on{Diff}(M)\,,$$ 
of shapes modulo reparameterizations. The quotient $B_i(M,\mathbb R^2)$ is an orbifold; see~\cite{Michor2006c}. 
Up to technicalities, equivalence classes $[c] 
\in B_i(M,\R^2)$ correspond to the image $c(M) \subset \R^2$ of the curve. Given a 
reparameterization invariant metric on $\on{Imm}(M,\R^2)$, it induces a metric on $B_i(M,\R^2)$, 
such that the projection map is a Riemannian submersion. See \cite{Bauer2012a} for details.
\end{remark*}
%Note that the identification of the quotient spaces $\on{Imm}/\on{Tra}$ and $\on{Imm}/\on{Mot}$ 
%that we have made in Sect.~\ref{Background:plane curves} are not invariant under 
%reparameterizations. A reparameterization invariant version can be obtained by taking integrated constraints:  
%\begin{align*}
% \on{Imm}(M,\R^2)/\on{Tra}&\cong\{c\in \on{Imm}(M,\R^2): \int_M c \ud s = 0\}\,,\\
% \on{Imm}(M,\R^2)/\on{Mot}&\cong\{c\in \on{Imm}(M,\R^2): \int_M c \ud s = 0 \text{ and } \int_M \al \ud s = 0\}\,.
%\end{align*}
%This would slightly change the formulas of the geodesic equations and of the inversions of the 
%maps $R$ presented in this article. Nevertheless, all results would remain valid. 

The following lemma provides a useful way to calculate the geodesic equation of such a metric.

\begin{lemma}\label{lem:geod_equa_general_metric}
Let $\mathcal C(M,\R^2)=\on{Imm}(M,\mathbb R^2)\cap V\subset  \on{Imm}(M,\R^2)$ be any of the sections 
mentioned in \ref{Background:plane curves}, where $V$ is the corresponding closed linear subspace 
of $C^\infty(M,\mathbb R^2)$.
Let $G=G^L$ be a weak Riemannian metric on $\mathcal C(M,\mathbb R^2)$ 
induced by an operator field $L$ via  
$$G^L_c(h,k)=\int_{M}\langle h,L_c k\rangle \ud s\,,$$
for $c\in \mathcal C(M,\R^2)$ and $h,k\in T_c \mathcal C(M,\R^2)$. 
Let $\bar L_c= L_c(\quad)\otimes \ud s$ and 
let $H_c(c_t,c_t)\in   (T_c \mathcal C(M,\R^2))'$ be defined via
$D_{c,m} (G_c(h,h))=\langle H_c(h,h),m\rangle_V$,
where $\langle \;,\;\rangle_V$ denotes the dual pairing between  
$V'$ and $V$.

Then the geodesic equation exists if and only if 
$\frac12 H_c(h,h)-(D_{c,h}L_c)(h)$ is in the image of  $\bar L_c$ 
and $(c,h)\mapsto \bar L_c\i\big(\frac12 H_c(h,h)-(D_{c,h}L_c)(h)\big)$ is 
smooth.
The geodesic equation can be written as:
$$
\boxed{\;
\begin{aligned}
p&=L_cc_t\otimes \ud s=\bar L_c(c_t)\\
p_t&=   \frac12 H_c(c_t,c_t)
\end{aligned}
\;}
\quad\text{  or }\quad
\boxed{\;
c_{tt} = \frac12 \bar L_c\i(H_c(c_t,c_t)) -\bar L_c\i\big(\p_t(\bar L_c)(c_t)\big)
\;}
$$
\end{lemma}

This lemma is an adaptation to the special situation here of the more general result 
\cite[Sect. 2.4]{Micheli2013}; there a \emph{robust weak Riemannian manifold} is one where the 
conclusion of this lemma holds together with a compatibility of the chart structure with the 
$G^L$-completions of all tangent spaces.   
Compare also with the version for diffeomorphism groups
\cite[Sect. 3.2]{Bauer2012d_preprint}.
In the sequel we shall compute $H_c(h,h)$ in many situations, but often we shall not check that it lies in 
the image of $\bar L$; the latter will follow directly from the representation of the metric via $R$-transforms. 

\begin{proof}
To  calculate the first variation of the energy we consider  a one-parameter family of curves 
$c:(-\ep,\ep)\times[0,1]\times M\to \R^2$ with fixed endpoints. The variational parameter will be denoted by
$\si\in (-\ep,\ep)$ and the time-parameter by $t \in [0, 1]$. We calculate:
\begin{align*}
\p_{\si}\frac12 \int_0^1 &G_c(c_t, c_t) \ud t = 
\frac12 \int_0^1 ({\p_{\si}} G_c)(c_t, c_t)\ud t
+ \int_0^1 G_c({\p_{\si}} c_t, c_t) \ud t 
\\&= 
\frac12 \int_0^1 \langle  H_c(c_t,c_t),c_{\si}\rangle_V\ud t
+ \int_0^1 G_c({\p_t} c_{\si}, c_t) \ud t 
\\&= 
\frac12 \int_0^1  \langle  H_c(c_t,c_t),c_{\si}\rangle_V \ud t
+ \int_0^1 \langle \bar L_c(c_t), {\p_t} c_{\si}\rangle_V \ud t 
\\&= 
\frac12 \int_0^1  \langle  H_c(c_t,c_t),c_{\si}\rangle_V \ud t
+ 0 - \int_0^1 \langle {\p_t}(\bar L_c(c_t)),  c_{\si}\rangle_V \ud t 
\\&= 
\int_0^1 G^L_c\Big(\bar L_c\i\big( \frac12  H_c(c_t,c_t) - \partial_t (\bar L_c c_{t})\big), c_\si\Big) \ud t
\\&= 
\int_0^1 G^L_c\Big(\bar L_c\i\big( \frac12  H_c(c_t,c_t) - (\partial_t \bar L_c)(c_{t})\big)- c_{tt}, c_\si\Big) \ud t\,.
\qedhere
\end{align*}
\end{proof}

\subsection{The $L^2$-metric on $C^\infty(M,N)$ and $H^k(M,N)$}

To show well-posedness of the geodesic equation we will need to work with the $L^2$-metric on Sobolev completions of manifolds of mappings. Here we summarise the necessary results.

Let $M$ be a compact manifold with volume form $\mu$ and $(N,g)$ a Riemannian manifold. We assume 
that both $M, N$ are finite dimensional. Let $k > \on{dim}M/2 + 1$. Then the $k$-th order Sobolev completion 
$H^k(M, N)$ of order $C^\infty(M,N)$ is a Hilbert 
manifold, and the tangent space is given by  
\[
T_q H^k(M,N) = \{ h \in H^k(M, TN) : \pi \o h = q \}\,.
\]
All results of this section also hold for $k=\infty$, i.e., for the Fr\'echet manifold $C^\infty(M,N)$.

We consider the weak Riemannian metric on $TH^k(M,N)$:
\begin{equation}
\label{eq:l2_metric}
G_q^{L^2}(h, k) = \int_M g_{q(x)}(h(x), k(x)) \ud \mu(x)\,.
\end{equation}
The following theorem summarizes the properties of $G^{L^2}$.

\begin{theorem}
\label{thm:ebin1970}
Let $k\in\mathbb N$ satisfy $k>\on{dim}M/2+1$ and let $G^{L^2}$ be defined by \eqref{eq:l2_metric}.
\begin{enumerate}
\item
$G^{L^2}$ defines a smooth weak Riemannian metric on $H^k(M,N)$.
\item
Let $\exp^g : TN \to N$ be the exponential map on $(N,g)$, defined on a neighbourhood of the zero-section. Then
\[
\exp^{L^2} : X \mapsto \exp^g \o X
\]
is the exponential map $\exp^{L^2} : T H^k(M,N) \to H^k(M,N)$ of the $G^{L^2}$-metric. 
It is a $C^\infty$-mapping defined on a neighbourhood of the zero-section.
\item
Let $\Xi^g : TN \to TTN$ be the geodesic spray of $(N,g)$. Then
\[
\Xi^{L^2} : X \mapsto \Xi^g \o X
\]
is the geodesic spray $\Xi^{L^2} : T H^k(M,N) \to TT H^k(M,N)$ of the $G^{L^2}$-metric. It is a $C^\infty$-mapping.
\item
Let $R^g : TN\x TN\x TN \to TN$ be the curvature tensor of $(N,g)$. Then
\[
R^{L^2} : (X, Y, Z) \mapsto R^g \o (X, Y, Z)
\]
is the curvature tensor $R^{L^2} : TH^k \x TH^k \x TH^k \to TH^k$ of the $G^{L^2}$-metric. It is a $C^\infty$-mapping.
\end{enumerate}
\end{theorem}

The proof of this theorem %is obvious since everything is a smooth vector field on a Hilbert manifold; it 
can be found in \cite[Thm.~9.1]{Ebin1970}, \cite[Cor.~9.3]{Ebin1970}, \cite[Prop. 2]{Bao1993} 
and \cite[Prop. 3.4]{Misiolek1993}. 

Given a submanifold of $H^k(M,N)$, the smoothness of the induced geodesic spray can be shown using 
the following theorem. 

\begin{theorem}
\label{thm:ebin1970_proj}
Let $k \in \mathbb N$ be as above and let $\mathscr M$ be a smooth submanifold of $H^k(M,N)$, such that the projection 
$\on{Proj}^{\mathscr M} : TH^k(M,N)\upharpoonright{\mathscr M} \to T\mathscr M$ is smooth. Then the 
geodesic spray of the metric $G^{L^2}$ on $\mathscr M$ is given by  
\[
\Xi^{\mathscr M} = \on{Proj}^{\mathscr M} \o \Xi^{L^2}
\]
and it is a smooth map.
\end{theorem}
This theorem is proven in \cite[Thm.~11.1]{Ebin1970}.

\section{A flat $H^2$-type metric}
\label{Sec:metric1}

\subsection{The metric and its geodesic equation}
In this section we will study an $H^2$-type metric on $\on{Imm}_{\on{conv}}(M,\R^2)/\on{Mot}$ -- the space of strictly convex curves. This metric has vanishing vanishing curvature for $M=[0,2\pi]$ and for $M=S^1$ the space is isometric to a codimension 2 submanifold of a flat space. The metric is 
given by   
\begin{equation}\label{metric1}
G_c(h,k)=\int_M \ka^{-3/2}\langle D^2_sh,n\rangle\langle D^2_sk,n\rangle + \langle D_sh,v\rangle\langle D_sk,v\rangle\ud s\,.
\end{equation}
Note that the metric is only defined for strictly convex curves, i.e., those satisfying $\ka > 
0$. It is sometimes more convenient to write $G$ via its associated operator $L$, 
%which is defined via  
\begin{equation*}
G_c(h,k) = \int_M \langle L_c h, k\rangle \ud s=\int_M \langle  h, L_ck\rangle \ud s\,,
\end{equation*}
%We will only calculate $L$ for  closed curves -- for open curves partial integration leads to 
%boundary terms thus making the expressions longer. Partial integration of \eqref{metric1} yields 
%\begin{align*}
%G_c(h,k) &= \int_{S^1} \ka^{-3/2}\langle D^2_sh,n\rangle\langle D^2_sk,n\rangle 
%+  \langle D_sh,v\rangle\langle D_sk,v\rangle\ud s
%\\
%&=\int_{S^1}    \langle  k, D_s^2 \big(\ka^{-3/2} \langle D_s^2 h, n \rangle n \big)\rangle  
%-\langle  k,D_s\big(\langle D_s h, v\rangle v\big)\rangle \ud s\,.
%\end{align*}
where  $L_c:T_c\on{Imm}_{\on{conv}}(S^1,\R^2)\to T_c\on{Imm}_{\on{conv}}(S^1,\R^2)$ is given by
\begin{equation*}\label{inertia_operator_metric1}
 L_ch =   D_s^2 \big(\ka^{-3/2} \langle D_s^2 h, n \rangle n \big)\rangle  -D_s\big(\langle D_s h, v\rangle v\big)\,.
\end{equation*}
% \begin{equation*}\label{inertia_operator_metric1}
%  L_c: \left\{
% \begin{array}{ccc}
% T_c\on{Imm}_{\on{conv}}(S^1,\R^2)&\to& T_c\on{Imm}_{\on{conv}}(S^1,\R^2)\,, \\
% h&\mapsto& D_s^2 \big(\ka^{-3/2} \langle D_s^2 h, n \rangle n \big)\rangle  -D_s\big(\langle D_s h, v\rangle v\big)\,.
% \end{array} \right.
% \end{equation*}
\begin{lemma}\label{lem:kernel:metric1}
The null space of the bilinear form $G_c(.,.)$ is spanned by constant vector fields and infinitesimal rotations, i.e.,
\begin{equation*}
 \on{ker}(G_c)=\left\{h\in T_c\on{Imm}_{\on{conv}}(M,\R^2): h=a + b.Jc,\, a\in\R^2, b \in \R\right\}\,. 
\end{equation*}
\end{lemma}
\begin{proof}
The null space of a symmetric bilinear form $A$ on $V\times V$  is the set
$$\on{ker}(A)=\{ v\in V: A(v,w)=0, \forall w\in V \}\,.$$
Thus for all $h$ in the kernel of $G_c$ we have $G_c(h,h)=0$. 
From this we see that for all $h\in \on{ker}(G_c)$ we have $\langle D_s h, v\rangle=0$ and   $\langle D_s^2 h, n \rangle=0$. 
If $h$ satisfies the two conditions then we also have $G_c(h,k)=0$ for all 
$k\in T_c\on{Imm}_{\on{conv}}(M,\R^2)$ and thus we see that the null space of  
$G_c$ consists exactly of these $h$ with $G_c(h,h)=0$. The condition $\langle D_s h, v \rangle =0$ yields $D_sh=b.n$, 
with $b\in C^{\infty}(M,\R)$, and the condition $\langle D_s^2 h, n \rangle =0$ implies that $b$ is constant. Taking the antiderivative of $D_sh$ we obtain the desired result.  
\end{proof}
As an immediate consequence of Lem.~\ref{lem:kernel:metric1} we obtain that $G_c$ is a weak 
Riemannian metric on $\on{Imm}_{\on{conv}}(M,\R^2)/{\on{Mot}}$.  
Its geodesic equation is giben by the following theorem.

\begin{theorem}
\label{thm:m1_geodesic_equation}
On the manifold $\on{Imm}_{\on{conv}}(M,\R^2)/\on{Mot}$ of plane parametrized curves modulo 
Euclidean motions $G_c(.,.)$ defines a weak Riemannian metric.   
For $M=S^1$ the geodesic equation is given by
\begin{align*}
p&=Lc_t\otimes \ud s\\&= D_s^2 \big(\ka^{-3/2} \langle D_s^2 c_t, n \rangle n \big)  
-D_s\big(\langle D_s c_t, v\rangle v\big)\otimes \ud s\,,\\
p_t&=D_s\Big(\frac12\langle D_sc_t,v\rangle^2v-\langle D_sc_t,n\rangle\langle D_sc_t,v\rangle n
-  D_s\big(\ka^{-3/2} \langle D_s c_t, n \rangle\langle D_s^2 c_t, n \rangle\big) v 
\\&
\phantom{=}\; +   \ka^{-3/2}\langle D_s^2 c_t, v \rangle\langle D_s^2 c_t, n \rangle n 
-\frac34D_s\Big(\ka^{-5/2} \langle D_s^2 c_t, n \rangle^2 n\Big) \Big)\otimes\ud s\,,
\end{align*}
with the additional constraint 
$$c_t\in T_c(\on{Imm}(S^1,\R^2)/\on{Mot})\cong\{h\in C^{\infty}(M,\R^2): h(0)=0, \langle D_sh(0),n(0)\rangle=0\}\,.$$
%\begin{align}
%\int_{S^1}c\ud s&=0, &\int_{S^1}c_t+\langle D_sc_t,v\rangle c\ud s=0\,,\label{metric1_mot_constr1}\\
%\int_{S^1}\al\ud s&=0, &\int_{S^1}\langle D_s c_t,n\rangle+\langle D_sc_t,v\rangle \al\ud s=0\,.\label{metric1_mot_constr2}
%\end{align}
\end{theorem}

The first part of the theorem applies to both $M=S^1$ and $M=[0,2\pi]$. The geodesic equation on 
the space of closed curves would additionaly contain a series of boundary terms arising from 
integrations by parts.

\begin{remark*}
Note that $L_c: 
T_c\left(\on{Imm}(M,\R^2)/{\on{Mot}}\right)\to T_c\left(\on{Imm}(M,\R^2)/{\on{Mot}}\right)$ is not 
an elliptic operator, since the highest derivative appears only in the normal direction. Thus we cannot 
apply the well-posedness results from \cite{Michor2007} or \cite{Bauer2011b}. 
Nevertheless we will show in Sects. 
\ref{Sec:metric1:open} and \ref{Sec:metric1:closed} that the geodesic equation is locally 
well-posed both on the space of open curves as well as the space of closed curves.   
\end{remark*}

\begin{proof}
To calculate the formula for the geodesic equation we  use Lem.~\ref{lem:geod_equa_general_metric}.
Using the variational formulas from Sect.~\ref{var} we calculate for a fixed vector field $h$ 
the variation of the components of the metric $G$: 
\begin{align*}
D_{c,m} \left( \langle D_s h,v \rangle \right) &=  
-\langle D_s m, v \rangle\langle D_sh,v\rangle+\langle D_sm, n \rangle\langle D_sh,n\rangle   
\\
D_{c,m} \left( \langle D_s^2 h, n \rangle \right) &=  
- \langle D_s h, n \rangle D_s \left(\langle D_s m, v \rangle \right) 
- \langle D_s^2 h, v \rangle \langle D_s m, n \rangle 
\\
&\phantom{=}\,\, - 2 \langle D_s^2 h, n \rangle\langle D_s m, v \rangle \,.
%\\
%D_{c,m}\left(|c_\th|\right) &= \langle D_s m, v \rangle |c_\th|\,\\
%T_c \ka .h &= \langle D_s^2 h, n \rangle - 2\ka \langle D_s h, v \rangle\,.
\end{align*}
Thus we obtain the following expression for the variation of the metric:
\begin{align*}
&D_{c,m}(G_c(h,h))= \int_{S^1} -\langle D_s m, v \rangle\langle D_sh,v\rangle^2
+2\langle D_sm, n \rangle\langle D_sh,n\rangle\langle D_sh,v\rangle
\\
& \qquad- 2\ka^{-3/2}\langle D_s h, n \rangle\langle D_s^2 h, n \rangle D_s \left(\langle D_s m, v \rangle \right) 
- 2\ka^{-3/2}\langle D_s^2 h, v \rangle\langle D_s^2 h, n \rangle \langle D_s m, n \rangle 
\\
& \qquad - \frac32 \ka^{-5/2}\langle D_s^2 m, n \rangle\langle D^2_s h, n \rangle^2  \ud s\,.
\end{align*}
We can now calculate $H_c(h,h)$ using a series of integrations by parts:
\begin{align*}
H_c(h,h)&= D_s\Big(\langle D_sh,v\rangle^2v-2\langle D_sh,n\rangle\langle D_sh,v\rangle n
- 2 D_s\big(\ka^{-3/2} \langle D_s h, n \rangle\langle D_s^2 h, n \rangle\big) v 
\\&
\phantom{=}\; +2   \ka^{-3/2}\langle D_s^2 h, v \rangle\langle D_s^2 h, n \rangle n 
-\frac32D_s\Big(\ka^{-5/2} \langle D_s^2 h, n \rangle^2 n\Big) \Big)\otimes \ud s\,.
\end{align*}

The existence of the geodesic equation, i.e., the invertibility of $\bar L$ will follow from the 
representation of the metric via the $R$-transform. The $R$-transform is an isometry from 
$\on{Imm}(M,\R^2)/\on{Mot}$ onto its image $\on{im}(R)$ and the image is a smooth submanifold of 
the space $(C^\infty(M,\R^2), G^{L^2})$ with a smooth orthogonal projection $TC^\infty(M,\R^2) \upharpoonright \on{im}(R) 
\to T\on{im}(R)$. Theorem~\ref{thm:ebin1970_proj} shows that the geodesic spray of the 
$L^2$-meric restricted  to $\on{im}(R)$ exists and is smooth and thus we can pull it back via $R$ to 
$\on{Imm}(M,\R^2)/\on{Mot}$.       
Hence the geodesic equation exists on $\on{Imm}(M,\R^2)/\on{Mot}$.
%The geodesic equation has to be amended with the constraints \eqref{metric1_mot_constr1} and  \eqref{metric1_mot_constr2},
%since we are identifying $\on{Imm}(M,\R^2)/{\on{Mot}}$ with the space
%\begin{align*}
% \on{Imm}(M,\R^2)/\on{Mot}\cong\{c\in \on{Imm}(M,\R^2): \int_M c \ud s = 0 \text{ and } \int_M \al \ud s = 0\}\,,
%\end{align*}
%as described in Sect.~\ref{Background:plane curves}.
\end{proof}

\subsection{The $R$-transform}\label{Sec:metric1:Rtrans}
Consider the map
\begin{equation}
\label{m1_rmap}
R: \left\{ \begin{array}{ccc}
\on{Imm}_{\on{conv}}(M, \R^2)/{\on{Mot}} &\to &C^\infty(M, \R^2) \\
c & \mapsto & \sqrt{|c'|} \,(2,4\ka^{1/4})
\end{array}\right. \,,
\end{equation}
and equip the space $C^\infty(M,\R^2)$ with the $L^2$-Riemannian metric,
\begin{equation}
\label{l2_met_r2_eukl}
G_q^{L^2}(h, k) = \int_{M} \langle h(\th), k(\th) \rangle \ud\th\,.
\end{equation}
Here $q \in C^\infty(M, \R^2)$ and $h, k \in T_q C^\infty(M, \R^2)$. Note that the Riemannian metric 
$G^{L^2}$ doesn't depend on the point $q$. The space $(C^\infty(M,\R^2), G^{L^2})$ is therefore a flat 
Riemannian manifold. 
%The reason for introducing the $R$-transform is that it maps the metric $G$ isometrically to $G^{L^2}$.   

\begin{theorem}
\label{thm:m1_rmap}
The map
\[
R\,:\, (\on{Imm}_{\on{conv}}(M, \R^2)/\on{Mot}, G) \to (C^\infty(M,\R^2), G^{L^2})
\]
is an injective isometry between weak Riemannian manifolds, i.e.,
\[
G_c(h, k) = G_{R(c)}^{L^2}(dR(c).h, dR(c).k)=\int_{M} \langle dR(c).h(\th), dR(c).k(\th)\rangle \ud \th
\]
 for $h,k \in T_c \on{Imm}_{\on{conv}}(M,\R^2)/\on{Mot}$.
\end{theorem}
\begin{proof}
Using formulas \eqref{var_length} and \eqref{var_ka} we obtain 
\begin{align*}
D_{c,h} \left( \sqrt{|c'|} \right) &= \tfrac 12 \langle D_s h, v \rangle \sqrt{|c'|} \\
D_{c,h} \left( \ka^{1/4} \sqrt{|c'|} \right) &=\tfrac 14 \ka^{-3/4} \langle D_s^2 h, n \rangle \sqrt{|c'|}\,,
\end{align*}
and thus the derivative of the $R$-transform is
\[
dR(c).h = \left( \langle D_s h, v \rangle \sqrt{|c'|},\, \ka^{-3/4} \langle D_s^2 h, n \rangle \sqrt{|c'|} \right)\,.
\]
Hence
\[
\langle dR(c).h, dR(c).h\rangle = \left(\langle D_c h, v \rangle^2  + \ka^{-3/2}\langle D_s^2 h, n \rangle^2 \right) |c'|\,,
\]
and the first statement of the theorem follows. 

To show that $R$ is injective we recall 
that one can recover the turning angle up to a constant  from the curvature function $\ka$ and the 
arclength $|c'|$ via $D_s\al=\ka$, or equivalently $$\al(\th)-\al(0) =  
\int_0^\th \ka|c'|\ud \th\,.$$ Choosing a different value for $\al(0)$ results in a 
rotation of the curve.  From the turning angle $\al$ and the arclength $|c'|$ one can reconstruct  
the curve $c$ up to a translation by integration: $c(\th)-c(0)=\int_0^{\th}e^{i\al}|c'|\ud \th$.
\end{proof}

\subsection{A motivation for this metric}\label{Sec:metric1:motivation}
The choice of the factor $\ka^{-3/2}$ in front of the term $\langle D_s^2 h, n\rangle^2$ in 
\eqref{metric1} appears to be arbitrary. One possibility to construct a second order Sobolev type 
metric on $\on{Imm}(M,\R^2)$ as the pullback of the flat $G^{L^2}$-metric, is to use an $R$-transform, that has a component of the form 
$R^j(c) = \sqrt{|c'|}f(\ka)$, where $f : \R \to \R$ is some smooth function. The derivative of 
$R^j$ with respect to $c$ is     
\[
dR^j(c)h = \sqrt{|c'|} \left( f'(\ka) \langle D_c^2 h, n \rangle 
- 2\ka f'(\ka) \langle D_c h, v \rangle + \tfrac 12 f(\ka) \langle D_c h, v \rangle \right)\,.
\]
The pullback metric would then contain a term
\[
\int_{S^1} \left( f'(\ka) \langle D_c^2 h, n \rangle - 2\ka f'(\ka) \langle D_c h, v \rangle 
+ \tfrac 12 f(\ka) \langle D_c h, v \rangle \right)^2 \ud s\,.
\]
In order to avoid cross-derivatives in the metric, the function $f(\ka)$ needs to satisfy
\[
\frac 12 f(\ka) = 2\ka f'(\ka)\,.
\]
Solutions to this ODE are given by
\[
f(\ka) = 
\begin{cases} 
C\ka^{1/4},&\ka > 0 \\
-C(-\ka)^{1/4},&\ka < 0
\end{cases}\,,
\]
with $C \in \R$ and the corresponding $R$-transform with $R^j(c) = 4 \sqrt{|c'|} \ka^{1/4}$ induces 
the following term in the metric 
\[
\int_{S^1} \ka^{-3/2} \langle D_c^2 h, n \rangle^2 \ud s\,.
\]
Thus the factor $\ka^{-3/2}$ is the unique choice to obtain a second order Sobolev type metric, 
which is flat on the space $\on{Imm}_{\on{conv}}(M,\R^2)/\on{Mot}$.

\subsection{The space of open curves}\label{Sec:metric1:open}
The image of the $R$-transform on the space $\on{Imm}_{\on{conv}}([0,2\pi],\R^2)$ of 
open convex curves is the set of all $\R_{>0}^{\;\,2}$-valued 
functions, which is an open set in $C^\infty([0,2\pi], \R^2)$.
\begin{theorem}\label{Thm:Rmapopen:metric1}
The $R$-transform
\[
R\,:\, \on{Imm}_{\on{conv}}([0,2\pi], \R^2)/\on{Mot}\to 
C^\infty([0,2\pi], \R_{>0}^{\;\,2})\underset{\on{open}}{\subset}C^\infty([0,2\pi], \R^2)
\]
is a diffeomorphism and its inverse is given by
\begin{equation}
\label{eq:m1_r_inv}
R\i : \left\{ \begin{array}{ccc}
\on{im}(R)_{\on{op}} & \to & \on{Imm}_{\on{conv}}([0,2\pi], \R^2)/\on{Mot} \\  
q & \mapsto & 2^{-2}\int_0^\th q_1^2 \on{exp}\left(i\int_0^{\si} 2^{-6}q_1^{-2}q_2^4 \ud \ta \right)\ud \th
\end{array} \right. \,.
\end{equation}
The space $(\on{Imm}_{\on{conv}}([0,2\pi], \R^2)/\on{Mot}, G)$ is a flat and
geodesically convex Riemannian manifold. 
\end{theorem}
\begin{proof}
It is shown in Thm.~\ref{thm:m1_rmap} that the $R$-transform is injective. The surjectivity will 
follow directly from the inversion formula. For the inversion formula note that we can reconstruct 
\begin{align*}
|c'| &= 2^{-2}q_1^2 &
\ka &= (\tfrac 12 q_1^{-1} q_2)^4 &
\al' &= 2^{-6} q_1^{-2} q_2^4
\end{align*}
and that
\[
c(\th) = c(0) + \int_0^\th \exp\left(i \al(0)\right)\exp\left(i\int_0^\si \al(\ta) \ud \ta\right) \ud \si\,.
\]
Since the initial conditions $c(0)$ and $\al(0)$ are unspecified, this determines the curve up to 
translations and rotations. The flatness follows from Thm.~\ref{thm:ebin1970}, since $\R_{>0}^2$ is 
flat. Given two curves $c_0$ and $c_1$, the minimizing geodesic connecting them is given by  
\[
c(t,\th) = R\i\left( tR(c_1) + (1-t)R(c_0) \right)(\th)\,,
\]
thus showing that $\on{Imm}_{\on{conv}}([0,2\pi], \R^2)/\on{Mot}$ is geodesically convex.
\end{proof}

\begin{remark*}
Since $\R_{>0}^{\;\,2}$ is geodesically incomplete, the same is true for the space 
$\on{Imm}_{\on{conv}}([0,2\pi], \R^2)/\on{Mot}$. A geodesic will leave the space, when it fails to 
be an immersion, i.e., $|c'(t,\th)| = 0$ for some $(t,\th)$, or when it stops being convex, i.e., 
$\ka(t,\th) = 0$. While the term $\ka^{-3/2} \langle D_s^2 h, n \rangle^2$ in the metric penalizes 
a curve from straightening out, it is not strong enough to prevent it.    
\end{remark*}

The $R$-transform allows us to give explicit formulas for geodesics and for the geodesic distance.
\begin{theorem}
\label{thm:m1_dist}
Given two curves $c_0, c_1 \in \on{Imm}_{\on{conv}}([0,2\pi], \R^2)/\on{Mot}$ the unique geodesic 
connecting them is given by 
\[
c(t,\th) = R\i\left( tR(c_1) + (1-t)R(c_0) \right)(\th)\,,
\]
and their geodesic distance is
\[
\on{dist}^G_{\on{op}}(c_0, c_1)^2 = 
\int_0^{2\pi} 16\left(\sqrt{|c_1'|}\ka_1^{1/4}-\sqrt{|c_0'|}\ka_0^{1/4}\right)^2
+4 \left(\sqrt{|c_1'|}-\sqrt{|c_0'|}\right)^2\ud \th\,.
\]
\end{theorem}

\begin{proof} This follows from Thm.~\ref{Thm:Rmapopen:metric1}.
\end{proof}
\begin{remark*}
The formula for the geodesic distance implies in particular that the following functions are Lipschitz continuous
\begin{align*}
\sqrt{|c'|}:
\left(\on{Imm}_{\on{conv}}([0,2\pi],\R^2)/\on{Mot},\on{dist}^G_{\on{op}}\right) 
&\to L^2([0,2\pi], \R) \\
\ka^{1/4}\sqrt{|c'|}:
\left(\on{Imm}_{\on{conv}}([0,2\pi],\R^2)/\on{Mot},\on{dist}^G_{\on{op}}\right) 
&\to L^2([0,2\pi], \R)\,.
\end{align*}
From the Lipschitz continuity of $\sqrt{|c'|}$ and the identity $\sqrt{\ell_c} = \| \sqrt{|c'|} \|_{L^2}$ we can then conclude that
\[
\sqrt{\ell_c} : \left(\on{Imm}_{\on{conv}}([0,2\pi],\R^2)/\on{Mot},\on{dist}^G_{\on{op}}\right) 
\to \R
\]
is also Lipschitz continuous.
An immediate consequence of this is the following lower bound for the geodesic distance:
$$\on{dist}^G_{\on{op}}(c_0,c_1)\geq 2\left|\sqrt{\ell_{c_1}}-\sqrt{\ell_{c_0}}\right|\,.$$
\end{remark*}

\subsection{The space of closed curves}\label{Sec:metric1:closed} 
When we restrict our attention to the space $\on{Imm}_{\on{conv}}(S^1,\R^2)/\on{Mot}$ of closed, 
strictly convex curves, the image of the $R$-transform is no longer an open subset of 
$C^{\infty}(S^1,\R^2)$.

Define the following function $H_{\on{cl}}: \on{Imm}_{\on{conv}}(S^1,\R^2)/\on{Mot} \to \R^2$, which measures 
how far away the preimage of $q$ is from being a closed curve:
\begin{align*}
H_{\on{cl}}(q)&=\int_0^{2\pi} q_1(\th)^2 \exp(i \al(q)(\th)) \ud \th\,, &
\al(q)(\th) &= 2^{-6} \int_0^\th q_1(\si)^{-2} q_2(\si)^4 \ud \si\,.
\end{align*}
The gradients of the components of $H_{\on{cl}}$ with respect to the $G^{L^2}$-metric are
\begin{subequations} \label{eq:m1_grad_hcl}
\begin{alignat}{1}
\on{grad}H_{\on{cl}}^1&=
\left(
\begin{array}{c}
2 q_1 \on{cos}\al(q)+2^{-5}  q_2^4q_1^{-3}\int_\th^{2\pi} q_1^2 \on{sin}\al(q)\ud\si \\
-2^{-4}  q_2^3q_1^{-2}  \int_\th^{2\pi} q_1^2 \on{sin}\al(q)\ud\si \\
\end{array}
\right)
\\
\on{grad}H^2_{\on{cl}}&=\left(
\begin{array}{c}
2 q_1 \on{sin}\al(q)-2^{-5}  q_2^4q_1^{-3}\int_\th^{2\pi} q_1^2 \on{cos}\al(q)\ud\si \\
2^{-4}  q_2^3q_1^{-2}  \int_\th^{2\pi} q_1^2 \on{cos}\al(q)\ud\si \\
\end{array}
\right)\,.
\end{alignat}
\end{subequations} 
The function $H_{\on{cl}}$ characterizes the image of the $R$-transform.

\begin{lemma}
\label{lem:m1_imcl_submf}
The image of $\on{Imm}_{\on{conv}}(S^1,\R^2)/\on{Mot}$ under the $R$-transform is 
$$\on{im}(R)_{\on{cl}}=\left\{q\in C^{\infty}(S^1,\R_{>0}^2): H_{\on{cl}}(q)=0  \right\}\,.$$
It is a splitting submanifold of $C^{\infty}(S^1,\R_{>0}^2)$ of codimension 2.

For $q \in \on{im}(R)_{\on{cl}}$ the orthogonal complement of $T_q \on{im}(R)_{\on{cl}}$ with 
respect to the $G^{L^2}$-metric is spanned by 
$\on{grad}H_{\on{cl}}^1(q), \on{grad}H_{\on{cl}}^2(q)$, given in \eqref{eq:m1_grad_hcl}.  
\end{lemma}

\begin{proof}
To characterize the image $\on{im}(R)_{\on{cl}}$ we recall the inversion formula \eqref{eq:m1_r_inv} from Thm.~
\ref{Thm:Rmapopen:metric1}. For $q\in C^{\infty}(\R_{>0}\times \R)$ it follows immediately that 
$R\i(q)$ is a closed curve if and only if $H_{\on{cl}}(q)=0$. 

We now show that $\on{im}(R)_{\on{cl}}$ is a splitting submanifold. 
Fix $q_0\in \on{im}(R)_{\on{cl}}$ and define the codimension 2 closed linear subspace
$$
T_{q_0}(\on{im}(R)_{\on{cl}}) := \big(\mathbb R.\on{grad}H_{\on{cl}}^1(q_0) + 
\mathbb R.\on{grad}H_{\on{cl}}^2(q_0)\big)^{\bot,G^{L^2}} \subset C^\infty(S^1,\mathbb R^2).
$$
We consider the affine isomorphism  
\[
A_{q_0} : \left\{ \begin{array}{ccc}
T_{q_0}(\on{im}(R)_{\on{cl}})\x \mathbb R^2 &\to &C^\infty(S^1,\mathbb R^2) \\
(h,x,y) &\mapsto & q_0 +  x.\on{grad}H_{\on{cl}}^1(q_0) + y.\on{grad}H_{\on{cl}}^2(q_0)
\end{array} \right. \,.
\]
% \]
% \begin{align*}
% A_{q_0}:&\; T_{q_0}(\on{im}(R)_{\on{cl}})\x \mathbb R^2 \to C^\infty(S^1,\mathbb R^2),
% \\
% A_{q_0}(h,x,y) :&= q_0 +  x.\on{grad}H_{\on{cl}}^1(q_0) + y.\on{grad}H_{\on{cl}}^2(q_0). 
% \end{align*}
Then the derivative $D_2(H_{\on{cl}}\o A_{q_0})(h,x,y)$ of the smooth mapping 
$$
H_{\on{cl}}\o A_{q_0}: T_{q_0}(\on{im}(R)_{\on{cl}})\x \mathbb R^2 \to \mathbb R^2
$$
is bounded and invertible for all $(h,x,y)$ near $(0,0,0)$, and thus 
$\on{im}(R)_{\on{cl}} = H_{\on{cl}}\i(0)$ is a smooth splitting submanifold by the 
implicit function theorem with parameters in a convenient vector space; see \cite[Lem. 1]{Teichmann01}, 
\cite{Hiltunen99} or \cite{Gloeckner06}.

Since $\on{grad}H_{\on{cl}}^1(q)$ and $\on{grad}H_{\on{cl}}^2(q)$ are linearly independent in 
$C^\infty(S^1,\mathbb R^2)$ for any $q\in C^\infty(S^1,\mathbb R_{>0}{}^2)$, they form a basis of 
$(T_q \on{im}(R)_{\on{cl}})^\perp$.  
\end{proof}

\begin{remark*}
Since the tangent space $T_q \on{im}(R)_{\on{cl}}$ has codimension 2, the orthogonal projection 
$\on{Proj}^{\on{im}} : TC^\infty(S^1,\R^2) \upharpoonright \on{im}(R)_{\on{cl}} \to T \on{im}(R)_{\on{cl}}$ 
is given by  
\begin{equation}
\label{eq:m1_orth_proj}
\on{Proj}^{\on{im}}(q).h = h - \langle h, w^1(q) \rangle w^1(q) -  \langle h, w^2(q) \rangle w^2(q)\,,
\end{equation}
where $w^1(q), w^2(q)$ is an orthonormal basis of $(T_q \on{im}(R)_{\on{cl}})^\perp$. In particular 
$\on{Proj}^{\on{im}}$ is smooth. 
\end{remark*}

Using the orthonormal basis $w^1(q), w^2(q)$ of $\on{im}(R)_{\on{cl}}^{\bot}$ and the Gau\ss{} 
equation one can calculate the curvature of $\on{im}(R)_{\on{cl}}$ and hence also of 
$\on{Imm}_{\on{conv}}(S^1,\R^2)/\on{Mot}$. This has been done for a different metric in \cite{Bauer2012b_preprint}.  

The geodesic equation on $\on{Imm}_{\on{conv}}(S^1,\R^2)/\on{Mot}$ is well-posed in appropriate 
Sobolev completions. We refer to Sect.~\ref{Sec:wp} and in particular to Sect.~\ref{Sec:wp_1_and_2} 
for a detailed discussion and proofs. The spaces $\on{Imm}^{j,k}(S^1,\R^2)$ are defined in 
\eqref{eq:custom_sob_scale}.   

\begin{theorem}
For $k\geq 2$ and initial conditions $(c_0, u_0) \in T\on{Imm}^{k+1,k+2}(S^1,\R^2)$, the 
geodesic equation has solutions in $\on{Imm}^{j+1,j+2}(S^1,\R^2)$ for each $2 \leq j \leq k$. The 
solutions depend $C^\infty$ on $t$ and the initial conditions and the domain of existence is 
independent of $j$.   

In particular for smooth initial conditions $(c_0, u_0) \in T\on{Imm}(S^1,\R^2)$ the geodesic 
equation has smooth solutions. 
\end{theorem}

% The main result is the following theorem, showing the local well-posedness of the geodesic equation 
% on $\on{im}(R)_{\on{cl}}$ and thus also of the geodesic equation on the manifold of 
% closed immersed curves.  
% \begin{theorem}
% The geodesic equation of the space $\big(\on{im}(R)_{\on{cl}},G^{L^2}\big)$ is given by: 
% \begin{align*}
% P_q^{\on{Nor}}(\ddot q) = 0,
% \end{align*}
% where  $P_q^{\on{Nor}}:C^\infty(S^1,\mathbb R^2) = 
% T_qC^\infty(S^1,\mathbb R_{>0}{}^2)\to (T_q\on{im}(R)_{\on{cl}})^\bot$
% is the orthonormal projection onto the $G^{L^2}$-normal bundle which can be explicitly described by 
% orthonormalizing the basis $\on{grad}(H_1)(q)$ and $\on{grad}(H_2)(q)$ of Lemma \ref{Sec:metric1:closed}.
% The geodesic equation is locally well posed on $\on{im}(R)_{\on{cl}}$.
% \end{theorem}
% \begin{remark*}
% Note that this proves the local well-posedness of the geodesic equation on $\on{Imm}_{\on{conv}}(S^1,\R^2)/\on{Mot}$.
% 

% \begin{proof} Since $P_q^{\on{Nor}}$ is an algebraic operator, even in $q$,
% the equation $P_q^{\on{Nor}}(\ddot q)(t,\th)=0$ is a smooth ODE on the Sobolev $H^k$-completion for 
% $k>1/2$, or even on the $C^k$-completion for $k\ge 0$, so it is well posed on these completions. 
% The time intervall of existence is determined by the time it runs out of the domain 
% $q(t,\th)\in \mathbb R_{>0}{}^2$, which is independent of $k$. 
% \end{proof}

\begin{remark*}
Since $\on{im}(R)_{\on{cl}} \subset \on{im}(R)_{\on{op}}$, the geodesic distance functions satisfy
\[
\on{dist}_{\on{op}}^G(c_0, c_1) \leq \on{dist}_{\on{cl}}^G(c_0, c_1)\,.
\]
Thus, the remark after Thm.~\ref{thm:m1_dist} also holds for the space 
$\on{Imm}_{\on{conv}}(S^1,\R^2)/\on{Mot}$ of closed curves, i.e., the functions $\sqrt{|c'|}$, 
$\ka^{1/4} \sqrt{|c'|}$ and $\sqrt{\ell_c}$ are Lipschitz continuous with respect to the geodesic 
distance.   
\end{remark*}

\section{The second metric}
\label{Sec:metric2}

\subsection{The metric and its geodesic equation}
The metric studied in the previous section is defined only for strictly convex curves. Consider the 
following related metric 
\begin{equation}
\label{metric2}
G_c(h,k) = \int_M \langle D_s h, v\rangle\langle D_s k, v\rangle  
+ \langle D_s^2 h, n \rangle \langle D_s^2 k, n \rangle \ud s\,,
\end{equation}
with $c\in \on{Imm}(M,\R^2)$ and $h,k\in T_c\on{Imm}(M,\R^2)$. This metric is defined for all curves. After integrating the expression of the metric by parts
\begin{align*}
G_c(h,k) &= \int_{S^1} \langle D_s h, v\rangle\langle D_s k, v\rangle  
+ \langle D_s^2 h, n \rangle \langle D_s^2 k, n \rangle \ud s 
\\
&=\int_{S^1}  -\langle  k,D_s\left(\langle D_s h, v\rangle v\right)\rangle  
+ \langle  k, D_s^2 \left(\langle D_s^2 h, n \rangle n \right)\rangle  \ud s\,.
\end{align*}
we obtain the associated operator field $L$ of the metric on the space $\on{Imm}(S^1,\R^2)$ of closed curves,
\[
L_c h = D_s^2 \left(\langle D_s^2 h, n \rangle n \right)
 -D_s\left(\langle D_s h, v\rangle v\right)\,.
\]
% \begin{equation*}\label{inertia_operator_metric2}
%  L_c: \left\{
% \begin{array}{ccc}
% T_c\on{Imm}(S^1,\R^2)&\to& T_c\on{Imm}(S^1,\R^2)\,, \\
% h&\mapsto& D_s^2 \left(\langle D_s^2 h, n \rangle n \right)
%  -D_s\left(\langle D_s h, v\rangle v\right)
% \end{array} \right.\,.
% \end{equation*}
The null space of $G_c$ is the same as in Sect.~\ref{Sec:metric1}.

\begin{lemma}\label{lem:kernel:metric2}
The null space of the bilinear form $G_c(.,.)$ is spanned by constant vector fields and 
infinitesimal rotations, i.e., 
\begin{equation*}\label{kernel_metric2}
 \on{ker}(G_c)=\left\{h\in T_c\on{Imm}(M,\R^2): h=a + b.Jc,\, a\in\R^2, b \in \R\right\}. 
\end{equation*}
\end{lemma}
\begin{proof}
See the proof of Lem.~\ref{lem:kernel:metric1}.
\end{proof}

It follows from this lemma that $G_c$ is a Riemannian metric on 
$\on{Imm}(M,\R^2)/\on{Mot}$. We will use Lem.~\ref{lem:geod_equa_general_metric} to calculate its 
geodesic equation.   

\begin{theorem}
On the manifold $\on{Imm}(M,\R^2)/\on{Mot}$ of plane parametrized curves modulo Euclidean motions 
$G_c(.,.)$ defines a weak Riemannian metric.   
For $M=S^1$ the geodesic equation is given by
\begin{align*}
p &=  L c_t\otimes \ud s=\left(D_s^2 \left(\langle D_s^2 c_t, n \rangle n \right)
-D_s\left(\langle D_s c_t, v\rangle v\right) \right)\otimes \ud s
\\
p_t &= \frac12 D_s\Big(\langle D_sc_t,v\rangle^2v -2\langle D_sc_t,n\rangle\langle D_sc_t,v\rangle n
- 2 D_s\big( \langle D_s c_t, n \rangle\langle D_s^2 c_t, n \rangle\big) v 
\\&\qquad\qquad\quad 
+2   \langle D_s^2 c_t, v \rangle\langle D_s^2 c_t, n \rangle n +3 \langle D_s^2 c_t, n \rangle^2 n \Big)\otimes \ud s
\,,
\end{align*}
with the additional constraint 
$$
c_t\in T_c(\on{Imm}(S^1,\R^2)/\on{Mot})\cong\{h\in C^{\infty}(M,\R^2): h(0)=0, \langle D_sh(0),n(0)\rangle=0\}\,.
$$
%\begin{align*}
%\int_{S^1}c\ud s&=0, &\int_{S^1}c_t+\langle D_sc_t,v\rangle c\ud s=0\,,\\
%\int_{S^1}\al\ud s&=0, &\int_{S^1}\langle D_s c_t,n\rangle+\langle D_sc_t,v\rangle \al\ud s=0\,.
%\end{align*}
\end{theorem}

\begin{remark*}
Similarily as in Sect.~\ref{Sec:metric1}, the operator $L_c: 
T_c\left(\on{Imm}(M,\R^2)/{\on{Mot}}\right)\to T_c\left(\on{Imm}(M,\R^2)/{\on{Mot}}\right)$  
is not an elliptic operator, since the highest derivative appears only in the normal direction. 
Again we cannot apply the well-posedness results from \cite{Michor2007} or \cite{Bauer2011b}. 
Instead we will show in Sects. 
\ref{Sec:metric2:open} and \ref{Sec:metric2:closed}, that the geodesic equation is locally 
well-posed both on the space of open and closed curves.     
\end{remark*}

\begin{proof}
To apply Lem.~\ref{lem:geod_equa_general_metric} we need to compute the $H_c$-gradient of the 
metric. Using the variational formulae from Sect.~\ref{var} we first calculate the variation of the 
metric  
% the variation of the components of the metric $G$: 
% \begin{align*}
% T_c \left( \langle D_s h,v \rangle \right)m &=  
% -\langle D_s m, v \rangle\langle D_sh,v\rangle+\langle D_sm, n \rangle\langle D_sh,n\rangle   
% \\
% T_c \left( \langle D_s^2 h, n \rangle \right)m &=  
% - \langle D_s h, n \rangle D_s \left(\langle D_s m, v \rangle \right) 
% - \langle D_s^2 h, v \rangle \langle D_s m, n \rangle 
% \\
% &\phantom{=}\; - 2 \langle D_s^2 h, n \rangle\langle D_s m, v \rangle 
% \\
% T_c\left(|c_\th|\right)m &= \langle D_s m, v \rangle |c_\th|\,.
% \end{align*}
% Thus we obtain the following expression for 
\begin{align*}
& D_{c,m}(G_c(h,h))= \int_{S^1} -\langle D_s m, v \rangle\langle D_sh,v\rangle^2
+2\langle D_sm, n \rangle\langle D_sh,n\rangle\langle D_sh,v\rangle
\\
& \qquad- 2\langle D_s h, n \rangle\langle D_s^2 h, n \rangle D_s \left(\langle D_s m, v \rangle \right) 
- 2\langle D_s^2 h, v \rangle\langle D_s^2 h, n \rangle \langle D_s m, n \rangle 
\\
& \qquad - 3 \langle D_s^2 h, n \rangle^2\langle D_s m, v \rangle  \ud s\,,
\end{align*}
and we integrate to obtain
\begin{align*}
H_c(h,h)&= D_s\Big(\langle D_sh,v\rangle^2v-2\langle D_sh,n\rangle\langle D_sh,v\rangle n
- 2 D_s\big( \langle D_s h, n \rangle\langle D_s^2 h, n \rangle\big) v 
\\&
\phantom{=}\; +2   \langle D_s^2 h, v \rangle\langle D_s^2 h, n \rangle n +3 \langle D_s^2 h, n \rangle^2 n \Big)\otimes \ud s
\,.
\end{align*}
Regarding the existence of the geodesic equation, see the proof of Thm.~\ref{thm:m1_geodesic_equation}. 
%The additional constraints are a consequence of the identification of the quotient space $\on{Imm}(S^1,\R^2)/\on{Mot}$ as a subset of $\on{Imm}(S^1,\R^2)$.
\end{proof}

\subsection{The $R$-transform}\label{Sec:metric2:Rtrans}
Consider the map
\begin{equation}
\label{m2_rmap}
R: \left\{ \begin{array}{ccc}
\on{Imm}(M, \R^2)/{\on{Mot}} &\to &C^\infty(M, \R^2) \\
c & \mapsto & (\sqrt{|c'|}, \ka|c'|^2)\end{array}\right. 
\end{equation}
On $\R^2$ we define the following Riemannian metric 
\[
g_{(q_1,q_2)} = \left( \begin{array}{cc}
4 & 0  \\
0 & q_1^{-6}
\end{array} \right)\,,
\]
and we equip the space $C^\infty(M, (\R^2,g))$ with the $L^2$-metric,
\begin{equation}
\label{l2_met_r3}
G_q^{L^2}(h, k) = \int_{M} g_{q(\th)}(h(\th), k(\th)) \ud \th\,.
\end{equation}
Here $q \in C^\infty(M, \R^2)$ and $h, k \in T_q C^\infty(M, \R^2)$. Note that as opposed to 
\eqref{l2_met_r2_eukl} the metric \eqref{l2_met_r3} does depend on the footpoint $q$. We will 
sometimes write $G_q^{L^2,g}$ to emphasize this dependence on the metric $g$. Since $(\R^2, g)$ is 
not a flat Riemannian manifold, neither is $(C^\infty(M, (\R^2, g)), G^{L^2})$.    

\begin{theorem}
With the metrics $g$ and $G^{L^2}$ defined as above, the map
\[
R\,:\, (\on{Imm}(M, \R^2)/\on{Mot}, G) \to (C^\infty(M, (\R^2, g)), G^{L^2})
\]
is an injective isometry between weak Riemannian manifolds, i.e.,
\[
G_c(h, k) = G_{R(c)}^{L^2}(dR(c).h, dR(c).k)=\int_{M} g_{R(c)(\th)}(dR(c).h(\th),  dR(c).k(\th)) \ud \th\,,
\] 
for $h,k \in T_c \on{Imm}(M,\R^2)/\on{Mot}$.
\end{theorem}

\begin{proof}
Using the formulas \eqref{var_length} and \eqref{var_ka} we obtain
\begin{align*}
D_{c,h} \left( \sqrt{|c'|} \right) &= \tfrac 12 \langle D_s h, v \rangle \sqrt{|c'|} \\
D_{c,h} \left( \ka |c'|^2\right) &= \langle D_s^2 h, n \rangle |c'|^2 
- 2\ka \langle D_s h, v \rangle |c'|^2 + 2 \ka \langle D_s h, v \rangle |c'|^2 \\
&= \langle D_s^2 h, n \rangle |c'|^2\,,
\end{align*}
and thus the derivative of the $R$-transform is
\[
dR(c).h = \left( \tfrac 12 \langle D_s h, v \rangle \sqrt{|c'|}, 
\langle D_s^2 h, n \rangle |c'|^2 \right)\,.
\]
Hence
\[
g_{R(c)}(dR(c).h, dR(c).h) = \left(\langle D_c h, v \rangle^2  + \langle D_s^2 h, n \rangle^2 \right) |c'|
\]
and the first statement of the theorem follows. Injectivity has already been shown in the proof of Thm.~\ref{thm:m1_rmap}.
\end{proof}

\subsection{The space of open curves}\label{Sec:metric2:open}
The image of the $R$-transform on the space $\on{Imm}([0,2\pi], \R^2)/\on{Mot}$ of open curves is the set of all 
$\R_{>0}\x \R$-valued functions, which is an open subset of $C^\infty([0,2\pi], \R^2)$. 

\begin{theorem}\label{Thm:Rmapopen:metric2}
 With the metrics $g$ and $G^{L^2}$ defined as above, the  $R$-transform
\[
R\,:\, \on{Imm}([0,2\pi], \R^2)/\on{Mot}\to C^\infty([0,2\pi], \R_{>0}\times\R)\underset{\on{open}}{\subset}
C^\infty([0,2\pi], \R^2)
\]
is a diffeomorphism. Its inverse is given by
\begin{equation*}
R\i: \left\{ 
\begin{array}{ccc}
C^\infty([0,2\pi], \R_{>0}\times\R)&\to& \on{Imm}([0,2\pi],\R^2)/\on{Mot}\\
q& \mapsto & \int_0^\th q_1^2 \on{exp}\left(i\int_0^\si q_2 q_1^{-2}\ud \tau \right)\ud \si
\end{array}\right. \,.
\end{equation*}
The space $(\on{Imm}([0,2\pi], \R^2)/\on{Mot}, G^{L^2,g})$ is a geodesically convex Riemannian manifold.
\end{theorem}

\begin{proof}
The characterization of the image and the inversion formula can be proven as in Thm.~\ref{Thm:Rmapopen:metric1} using
\begin{align*}
|c'|&= q_1^2 &
\ka &= q_1^{-4} q_2 &
\al' &= q_1^{-2}q_2\,.
\end{align*}
It is shown in Sect. \ref{metric-2dim} that $(\R_{>0}\x \R, g)$ is geodesically convex and that the 
geodesic connecting two points is unique. Given two curves $c_0$ and $c_1$, the minimizing geodesic 
connecting them is described in \eqref{eq:m2_geodesic_bvp}. Thus $\on{Imm}([0,2\pi],\R^2)/\on{Mot}$ is 
geodesically convex.   
\end{proof}

%\begin{remark*}
%Given two curves $c_0, c_1$ the geodesic connecting them is given by
%\begin{equation}
%\label{eq:m2_geodesic_bvp}
%c(t,\th) = R\i(q(t)(\th)\,,
%\end{equation}
%where for each $\th \in [0,2\pi]$ the curve $t \mapsto q(\cdot,\th)$ is the geodesic connecting $R(c_0)(\th)$ and $R(c_1)(\th)$.
%

\begin{remark*}
Since $\R_{>0}\x \R$ is geodesically incomplete (see Sect.~\ref{metric-2dim}), so is the space 
$\on{Imm}([0,2\pi], \R^2)/\on{Mot}$. A geodesic will leave the space, when it fails to be an 
immersion, i.e., when $c'(t,\th)=0$ for some $(t,\th)$.  
\end{remark*}

The $R$-transform allows us to give formulas for geodesics and a lower bound for the geodesic distance.
\begin{theorem}\label{metric2:geod_dist}
Given two curves $c_0, c_1 \in \on{Imm}_{\on{conv}}([0,2\pi], \R^2)/\on{Mot}$ the unique geodesic connecting them is given by
\begin{equation}
\label{eq:m2_geodesic_bvp}
c(t,\th) = R\i(q(t))(\th)\,,
\end{equation}
where for each $\th \in [0,2\pi]$ the curve $t \mapsto q(\cdot,\th)$ is the geodesic connecting $R(c_0)(\th)$ and $R(c_1)(\th)$.
The geodesic distance is bounded from below by
%\[
%\on{dist}((x_0,y_0), (x_1,y_1)) \geq
%2 \sqrt{|x_0 - x_1|^2 + \frac{|y_0-y_1|^2}{2^{1/8} \left(x_0^4 + x_1^4 + \tfrac 1{2A} |y_0-y_1|\right)^{3/2}}}\,.
%\]
\begin{multline*}
\on{dist}^G_{\on{op}}(c_0, c_1)^2\\ \geq
\int_0^{2\pi} \frac{2^{15/8}\left(|c_1'|^2\ka_1-|c_0'|^2\ka_0\right)^2}{\left(|c_0'|^4 + |c_1'|^4 + \tfrac 1{2A}\big|\left|c_1'\right|^2\ka_1-|c_0'|^2\ka_0\big|\right)^{3/2}}+16 \left(\sqrt{|c_1'|}-\sqrt{|c_0'|}\right)^2\ud \th\,,
\end{multline*}
where $A$ is the constant 
\[
A = \int_0^1 \frac{z^6 \ud x}{\sqrt{1 - z^6}} = 0.30358...\,.
\]
\end{theorem}
\begin{proof}
The proof of this theorem follows directly from the analysis of the finite dimensional metric $g$, see Sect.~\ref{metric-2dim}.
\end{proof}

\begin{remark*}
The formula for the geodesic distance implies in particular that the function
\begin{align*}
\sqrt{|c'|}:
\left(\on{Imm}([0,2\pi],\R^2)/\on{Mot},\on{dist}^G_{\on{op}}\right) 
&\to L^2([0,2\pi], \R)
\end{align*}
is Lipschitz continuous, Since $\sqrt{\ell_c} = \| \sqrt{|c'|} \|_{L^2}$ we can then conclude that
\[
\sqrt{\ell_c} : \left(\on{Imm}_{\on{conv}}([0,2\pi],\R^2)/\on{Mot},\on{dist}^G_{\on{op}}\right) 
\to \R
\]
is also Lipschitz continuous. This implies the following lower bound for the geodesic distance:
$$\on{dist}^G_{\on{op}}(c_0,c_1)\geq 4\left|\sqrt{\ell_{c_1}}-\sqrt{\ell_{c_0}}\right|\,.$$
\end{remark*}
%An immediate consequence of this is the following lower bound for the geodesic distance:
%$$\on{dist}^G_{\on{op}}(c_0,c_1)\geq 16(\sqrt{\ell_{c_1}}-\sqrt{\ell_{c_0}})^2\,.$$

The space $(\on{Imm}([0,2\pi], \R^2)/\on{Mot}, G^{L^2})$ is not flat. The representation via 
$R$-transform allows us to calculate its curvature. 

\begin{theorem}
The sectional curvature of the plane spanned by $h,k$ is given by
\[
\mathrm{k}_c(P(h,k)) = \frac{-3 \int_{0}^{2\pi} \left( \langle D_s h, v \rangle \langle D_s^2 k, n \rangle -
\langle D_s k, v \rangle \langle D_s^2 h, n \rangle \right)^2 \ud s}
{G_c(h,h) G_c(k,k) - G_c(h,k)^2}\,,
\]
with $h,k \in T_c \on{Imm}([0,2\pi], \R^2)/\on{Mot}$. In particular the sectional curvature is non-positive.
\end{theorem}
\begin{proof}
The sectional curvature of $(C^\infty([0,2\pi],\R^2), G^{L^2})$ is the integral over the pointwise 
sectional curvatures, see Thm.~\ref{thm:ebin1970}. The statement now follows from the curvature 
formulas for $g$ given in Sect.~\ref{metric-2dim}.  
\end{proof}

\subsection{The space of closed curves}\label{Sec:metric2:closed} 
The image of the space $\on{Imm}(S^1,\R^2)/\on{Mot}$ of closed curves under the $R$-transform is 
not an open subset of $C^\infty(S^1,\R^2)$. 

Consider the following function $H_{\on{cl}} : \on{Imm}(S^1,\R^2)/\on{Mot} \to \R^2$, which measures how far 
away the preimage of $q$ is from being closed:
\begin{align*}
H_{\on{cl}}(q)&=\int_0^{2\pi} q_1(\th)^2 \exp(i\al(q)(\th)) \ud \th\,, &
\al(q)(\th)&=\int_0^\th q_2(\si) q_1(\si)^{-2}\ud \si \,.
\end{align*}
The gradients of the components of $H_{\on{cl}}$ with respect to the $G^{L^2}$-metric are
\begin{subequations} \label{eq:m2_grad_hcl}
\begin{alignat}{1}
 \on{grad}(H_{\on{cl}}^1)&=
\begin{pmatrix}
\frac12 q_1 \on{cos}\al(q)+\frac12 q_2q_1^{-3}\int_\th^{2\pi} q_1^2 \on{sin}\al(q)\ud\si\\
-q_1^4\int_\th^{2\pi} q_1^2 \on{sin}\al(q)\ud\si   
\end{pmatrix} \\
 \on{grad}(H_{\on{cl}}^2)&=
\begin{pmatrix}
\frac12 q_1 \on{sin}\al(q)-\frac12 q_2q_1^{-3}\int_\th^{2\pi} q_1^2 \on{cos}\al(q)\ud\si\\
+q_1^4\int_\th^{2\pi} q_1^2 \on{cos}\al(q)\ud\si 
\end{pmatrix}
\,.
\end{alignat}
\end{subequations} 
The function $H_{\on{cl}}$ permits us to characterize the image of the $R$-transform.

\begin{lemma}
\label{lem:m2_imcl_submf}
The image of $\on{Imm}(S^1,\R^2)/\on{Mot}$ under the $R$-transform is given by 
$$\on{im}(R)_{\on{cl}}=\left\{q\in C^{\infty}(S^1,\R_{>0}\x\R): H_{\on{cl}}(q)=0  \right\}\,.$$
It is a splitting submanifold of $C^{\infty}(\R_{>0}\x\R)$ of codimension 2.

For $q \in \on{im}(R)_{\on{cl}}$, the orthogonal complement of $T_q \on{im}(R)_{\on{cl}}$ with 
respect to the $G^{L^2}$-metric is spanned by 
$\on{grad}H_{\on{cl}}^1(q), \on{grad}H_{\on{cl}}^2(q)$, given in \eqref{eq:m2_grad_hcl}.  
\end{lemma}

\begin{proof}
Mutatis mutandis, we can reuse the proof of Lem.~\ref{lem:m1_imcl_submf}.
\end{proof}

\begin{remark*}
The orthogonal projection $\on{Proj}^{\on{im}} : TC^\infty(S^1,\R^2) 
\upharpoonright \on{im}(R)_{\on{cl}} \to T \on{im}(R)_{\on{cl}}$ is again given by 
\eqref{eq:m1_orth_proj} and it is a smooth map.  
\end{remark*}

The geodesic equation on $\on{Imm}(S^1,\R^2)/\on{Mot}$ is well-posed in appropriate Sobolev 
completions. We refer to Sect.~\ref{Sec:wp} and in particular to Sect.~\ref{Sec:wp_1_and_2} for a 
detailed discussion and proofs. The spaces $\on{Imm}^{j,k}(S^1,\R^2)$ are defined in 
\eqref{eq:custom_sob_scale}.   

\begin{theorem}
For $k\geq 2$ and initial conditions $(c_0, u_0) \in T\on{Imm}^{k+1,k+2}(S^1,\R^2)$, the 
geodesic equation has solutions in $\on{Imm}^{j+1,j+2}(S^1,\R^2)$ for each $2 \leq j \leq k$. The 
solutions depend $C^\infty$ on $t$ and the initial conditions and the domain of existence is 
independent of $j$.   

In particular for smooth initial conditions $(c_0, u_0) \in T\on{Imm}(S^1,\R^2)$ the geodesic 
equation has smooth solutions. 
\end{theorem}

\begin{remark*}
Since $\on{im}(R)_{\on{cl}} \subset \on{im}(R)_{\on{op}}$, the geodesic distance functions satisfy
\[
\on{dist}_{\on{op}}^G(c_0, c_1) \leq \on{dist}_{\on{cl}}^G(c_0, c_1)\,.
\]
Therefore the remark after Thm.~\ref{metric2:geod_dist} also holds for the space 
$\on{Imm}(S^1,\R^2)/\on{Mot}$ of closed curves, i.e., the functions $\sqrt{|c'|}$ and 
$\sqrt{\ell_c}$ are Lipschitz continuous with respect to the geodesic  distance.   
\end{remark*}

\subsection{Analysis of the Riemannian manifold $(\mathbb R_{>0}\times \mathbb R,g)$}
\label{metric-2dim}
The metric is $g=4dx^2 +x^{-6}dy^2 = \si^1\otimes\si^1+\si^2\otimes\si^2$, where $\si^1= 
2dx$, $\si^2 = x^{-3}dy$ is an orthonormal coframe. 

First we note that the Levi-Civita connection has a 2-dimensional symmetry group; its 
transformations map geodesics to geodesics: 
$$
T_a\binom{x}{y} = \binom{x}{y+a}\,,\quad a\in \mathbb R\,;\qquad 
H_r\binom{x}{y} = \binom{rx}{r^4y}\,, \quad r\in \mathbb R_{>0}\,.
$$
Namely, each $T_a$ and each reflection $y\mapsto -y$ is an isometry. 
On the other hand, $H_r^*g = 4d(rx)^2 + (rx)^{-6}d(r^4y)^2= r^2g$, and a constant multiple of a 
Riemannian metric has the same geodesics. 
Since $H_r\,T_a\,H_r\i= T_{r^4\,a}$, this is the semidirect product 
$\mathbb R_{>0}\rtimes \mathbb R$ with multiplication
$(T_{a_1},H_{r_1})(T_{r_2},H_{r_2}) = (T_{a_1+r_1^4a_2}H_{r_1r_2})$.

We have $d\si^1=0$ and $d\si^2= -3 x^{-4}\,dx\wedge dy = 
-(3/2) x^{-1}\si^1\wedge \si^2$.
Using Cartan's structure equation we compute the connection form 
$\om\in \Om^1(\mathbb R_{>0}\x \mathbb R;\mathfrak{o}(2))$, 
$$
\begin{aligned}
-d\si^1&=0+\om^1_2\wedge \si^2 =0 
\\
-d\si^2&=\om^2_1\wedge \si^1 +0= \tfrac3{2x}\si^1\wedge \si^2 
\end{aligned}
\implies
\om = 
\begin{pmatrix}
0 &  \tfrac3{2x^4}dy
\\
-\tfrac3{2x^4}dy &0 
\end{pmatrix}
$$
The \emph{curvature matrix} and the \emph{Gauss curvature} are then:
$$
\Om = d\om+\om\wedge \om = 
\begin{pmatrix}
0 & -\tfrac3{x^2}\si^1\wedge \si^2
\\
\tfrac3{x^2}\si^1\wedge \si^2 & 0
\end{pmatrix},
\qquad \on{scal}(g) = -\frac3{x^2}\,.
$$
This already implies that there are no conjugate points along any geodesic and that the Riemannian 
exponential mapping centered at each point is a diffeomorphism onto its image.

We use the dual orthonormal frame	$s_1=\tfrac12\p_{x}$, $s_2=x^3\p_{y}$. Then the velocity of 
a curve $q(t)$ is  
$\dot q = \dot x\p_{x} + \dot y\p_{y} = 2\dot x(s_1\o q) + \dot y x^{-3}(s_2\o q)$, 
and 
the \emph{geodesic equation} is as follows:
$$
\begin{aligned}
0 &= \nabla_{\p_t}\dot q = \nabla_{\p_t}\big(2\dot x(s_1\o q) + \dot y x^{-3}(s_2\o q)\big) 
\\&
=(\ddot x + \tfrac34 x^{-7}\dot y^2\big)\p_{x} 
+ \big(\ddot y -6 x^{-1}\dot x\dot y\big)\p_{y}
\end{aligned}
\qquad\text{  or }\quad 
\boxed{\quad 
\begin{aligned}
\ddot x &= - \tfrac34 x^{-7}\dot y^2
\\ 
\ddot y &= 6 x^{-1}\dot x\dot y
\end{aligned} \quad}
$$
Note that $(x_0 + t \dot x_0,y_0)$ are incomplete geodesics. Hence, if $\dot y(t)=0$ for 
some $t$, then for all $t$. If $\dot y> 0$, it never can change sign, always $\ddot x<0$ and  
the geodesic, which is curving always to the left, leaves the space for $t\to \infty$. The case
$\dot y<0$ is mapped to $\dot y>0$ by  the reflection $(x,y)\mapsto (x,-y)$. 
 
Now we eliminate time from the geodesic equation. %Suppose $\dot y>0$.
The second equation can be rewritten as  
\begin{align*}
\frac{\ddot y}{\dot y} = 6 \frac{\dot x}{x} \;&\iff\;
\log(|\dot y|) = 6\log(x) +\log(|\dot y_0|) -6\log(x_0) 
\\&\iff\;
\dot y = C_1x^6\quad\text{  where }\quad C_1=\frac{\dot y_0}{x_0^6}.
\end{align*}
Inserting this, the first equation becomes  
\begin{align*}
\ddot x = - \tfrac{3}{4}C_1^2 x^5 \;&\iff\;
2\ddot x\dot x = - \tfrac{3}{2}C_1^2 x^5\dot x
\;\iff\;
\dot x^2 = -\tfrac{1}{4} C_1^2 x^6 + \dot x_0^2
\\&\iff\;
\dot x = \on{sign}(\dot x_0)\big(C_2^2 -\tfrac{1}{4} C_1^2 x^6\big)^{1/2}
\quad\text{  where }\quad C_2=\dot x_0.
\end{align*}
Therefore,
\begin{align*}
\frac{dy}{dx} &= \frac{\dot y}{\dot x} = \frac{\on{sign}(\dot x_0) C_1 x^6}{(C_2^2 -\tfrac14 C_1^2 x^6)^{1/2}}
= \frac{C_1 x^6}{C_2(1 - |\frac{C_1}{2C_2}|^2x^6)^{1/2}}
= \frac{2 C^3 x^6}{(1 - (|C|x)^6)^{1/2}}\,,
\\
C&=\Big(\frac{C_1}{2C_2}\Big)^{1/3}=\Big(\frac{\dot y_0}{2x_0^6\dot x_0}\Big)^{1/3}\,.
\end{align*}
Note that 
$\on{sign}(C) = \on{sign}\Big(\frac{\dot y_0}{\dot x_0}\Big)$.
We get $y$ as a function of $x$, namely
$$
y(x) = y_0 + \int_{x_0}^{x} \frac{2 C^3 z^6\ud z}{(1 - (|C|z)^6)^{1/2}}
= y_0 + \frac{2\on{sign}(C)}{|C|^4}\int_{|C|x_0}^{|C|x} \frac{z^6\ud z}{\sqrt{1 -  z^6}}
$$
for $0<x_0,x\le 1/|C|$. Thus the trajectory of a geodesic is given by:
\begin{equation}\label{eq:2dimGeodesic}
\boxed{\;
\begin{aligned}
y(x) &= y_0 +\on{sign}\Big(\frac{\dot y_0}{\dot x_0}\Big)\frac2{|C|^4}\big(F(|C|x) - F(|C|x_0)\big) 
\\&\quad
\text{  for }0\le x_0 \le x\le 1/|C|\,,\quad\text{  where}
\\
F(u) &= \int_{0}^{u} \frac{z^6\ud z}{\sqrt{1 -  z^6}}
\qquad\text{  with }
\qquad A:= F(1) = 0.30358...\,.
\end{aligned}
\;}
\end{equation}
Let us now assume that $\dot y_0>0$ and $\dot x_0>0$. Then 
\begin{align*}
y(0)&= y_0 - \frac2{|C|^4}F(|C|x_0)\,,
 &\qquad
y(1/|C|)&= y(0) + \frac{2 F(1)}{|C|^4}\,,
\\
\frac{dy}{dx}(0) &= 0\,, & \frac{dy}{dx}(1/{|C|})&=\infty\,.
\end{align*}
The reflection $y\mapsto -y + 2y(1/|C|)$ maps this geodesic to its other half which returns to 
$x=0$.

Given two points $(x_0>0,y_0)$ and $(x_1>0,y_1)$, without loss satisfying $x_0<x_1$ and $y_0<y_1$, 
we shall now \emph{determine the unique connecting geodesic trajectory}. By our assumption we have 
$\dot y_0>0$ and $\dot x_0>0$, thus $\on{sign}(C)=\on{sign}(\dot y_0/\dot x_0) = 1$. 
\newline
\emph{Case 1.}
The two points lie on the right travelling branch of the geodesic arc if we can find 
$C>0$ such that 
$$
y_1 = y_0 + \frac2{C^4}\big(F(Cx_1)-F(Cx_0)\big),\qquad 0< x_0<x_1\le \frac1C.
$$
The function $f(C)=f_{x_0,x_1}(C):=\frac2{C^4}\big(F(Cx_1)-F(Cx_0)\big)$ is monotone increasing 
in $0<C\le \frac1{x_1}$, since (note that $F(u)$ is flat for $u\searrow 0$)
$$
f'(C)=-\frac{8}{C^5}\big(F(Cx_1)-F(Cx_0)\big)
+2{C^2}(\frac{x_1^7}{\sqrt{1-(Cx_1)^6}} - \frac{x_0^7}{\sqrt{1-(Cx_0)^6}})>0.
$$
This is the case if and only if 
$$
0<y_1-y_0\le f_{x_0,x_1}(1/x_1)=2x_1^4(F(1)-F(x_0/x_1)) =  2x_1^4\int_{x_0/x_1}^1 
\frac{z^6\ud z}{\sqrt{1-z^6}}\,.
$$
\emph{Case 2.} The point $(x_0,y_0)$ lies on the right travelling branch, but $(x_1,y_1)$ is reached 
only after passing the apex on the left travelling branch. We know that $y_1-y_0> 
f_{x_0,x_1}(1/x_1)$.
We have to find the apex $(\bar x,\bar y)$ with 
$x_0<x_1<\bar x$, $y_0<\bar y< y_1$,  and $C=1/\bar x$. The apex is given by
$$
\bar y := y_0 + 2\bar x^4\big(F(1)-F(x_0/\bar x)\big)
= y_1 - 2\bar x^4\big(F(1)-F(x_1/\bar x)\big).
$$
So we have to solve 
$$
y_1-y_0 = 2\bar x^4\big(2F(1)-F(x_0/\bar x)-F(x_1/\bar x)\big)
$$
for $\bar x=1/C > x_1$. Since the right hand side is monotone increasing in $C=1/\bar x$ there is a 
unique solution.

This implies that $(\mathbb R_{>0}\x \mathbb R,g)$ is geodesically convex, and there is a unique 
geodesic connecting any two points.

Since we need it later, we describe the \emph{Riemann curvature tensor}, for 
$q=(q_1,q_2)\in \mathbb R_{>0}\times \mathbb R$ and $h,k,\in \mathbb R^2$:
\begin{align*}
(\si^1\wedge \si^2)_q(h,k)&= \frac2{q_1^3}(h_1k_2-k_1h_2) = 
\frac2{q_1^3}\det(h,k)=\frac2{q_1^3}\langle h,ik \rangle
\\
\mathcal R^g_q(h,k)\ell &= s_1\Om^1_2(h,k)\si^2(\ell) + s_2\Om^2_1(h,k)\si^1(\ell)
\\&
= \begin{pmatrix} 
0 & -\frac3{q_1^8}\langle h,Jk \rangle \\
\frac{12}{q_1^3}\langle h,Jk \rangle	 & 0
\end{pmatrix}
\binom{\ell_1}{\ell_2}
\\
g_q(\mathcal R^g_q(h,k)k,h) & = \on{scal}(g)(q). \bigl(\|h\|^2_{g_q}\|k\|^2_{g_q} - g_q(h,k)^2\bigr)
\\&
=\frac{-3}{q_1^2}\Big(\big(4h_1^2+\frac{h_2^2}{q_1^6}\big)\big(4k_1^2+\frac{k_2^2}{q_1^6}\big)
-\big(4h_1k_1+\frac{h_2k_2}{q_1^6}\big)^2\Big)
\\&
=\frac{-12}{q_1^8}\left( h_1k_2 - h_2k_1\right)^2
\end{align*}

\begin{lemma}
The geodesic distance admits the following lower bound
\[
\on{dist}((x_0,y_0), (x_1,y_1)) \geq
2 \sqrt{|x_0 - x_1|^2 + \frac{|y_0-y_1|^2}{2^{1/8} \left(x_0^4 + x_1^4 + \tfrac 1{2A} |y_0-y_1|\right)^{3/2}}}\,.
\]
\end{lemma}

\begin{proof}
Let $\ga(t)$ be the geodesic connecting $(x_0,y_0)$ and $(x_1,y_1)$. The distance between the 
points, where the continuation of $\ga$ intersects the $y$-axis is less than 
\[
2A \left(x_0^4 + x_1^4\right) + |y_0-y_1|\,.
\]
Thus
\[
\ga^1(t) \leq 2^{-1/4} \left( x_0^4 + x_1^4 + \tfrac 1{2A} |y_0-y_1| \right)^{1/4}\,.
\]
Define
\[
r = 2^{-1/12}\left( x_0^4 + x_1^4 + \tfrac 1{2A} |y_0-y_1| \right)^{1/4}\,,
\]
and rescale the point by $(\ol x_i, \ol y_i) = (rx_i, r^4 y_i)$. From the symmetries of the 
Levi-Civita connection we see that the geodesic connecting $(\ol x_0,\ol y_0)$ and 
$(\ol x_1,\ol y_1)$ is given by $(\ol \ga^1, \ol \ga^2) = (r\ga^1, r^4\ga^2)$. Hence  
\[
\ol \ga^1(t) = r \ga^1(t) \leq 2^{-1/3} = 4^{-1/6}\,.
\]
On the set $\{ (x, y) \,:\, x \leq 4^{-1/6} \}$ we have
\[
g(x,y) = 4dx^2 + x^{-1/6}dy^2 \geq 4dx^2 + 4dy^2
\]
and thus
\begin{align*}
\on{dist}((x_0,y_0), (x_1,y_1)) &= \on{Len}^g(\ga) = r\i \on{Len}^g(\ol \ga) 
\geq r\i \on{Len}^{\on{Eucl}}(\ol \ga) \\
&\geq r\i \sqrt{|\ol x_0 - \ol x_1|^2 + |\ol y_0-\ol y_1|^2} \\
&= \sqrt{|x_0 - x_1|^2 + r^6 |y_0-y_1|^2}\,.
\end{align*}
This concludes the proof.
\end{proof}

\section{The third metric}
\label{Sec:metric3}
\subsection{The metric and its geodesic equation}
The metrics studied in the previous two sections both had Euclidean motions in their kernel. In 
this section we will study the following metric  
\begin{equation}
\label{metric3}
G_c(h,k) = \int_M \langle D_s h, D_s k \rangle  + \langle D_s^2 h, n \rangle \langle D_s^2 k, n \rangle \ud s\,,
\end{equation}
whose kernel consists only of translations. We can obtain it from \eqref{metric2} by adding the term 
$\langle D_s h, n \rangle \langle D_s k, n \rangle$. 
We shall study this metric only on the space of closed curves because the $R$-transform is not 
better behaved on the space of open curves.

The associated operator field $L$ is given by
\[
L_c h = D_s^2 \big(\langle D_s^2 h, n \rangle n \big)  -D^2_sh\,.
\]
% \begin{equation*}\label{inertia_operator_metric3}
%  L_c: \left\{
% \begin{array}{ccc}
% T_c\on{Imm}(S^1,\R^2)&\to& T_c\on{Imm}(S^1,\R^2)\,, \\
% h&\mapsto&D_s^2 \big(\langle D_s^2 h, n \rangle n \big)\rangle  -D^2_sh
% \end{array} \right.\,.
% \end{equation*}
\begin{lemma}\label{lem:kernel:metric3}
The null space of the bilinear form $G_c(.,.)$ is spanned by constant vector fields, i.e.,
\begin{equation*}\label{kernel_metric3}
 \on{ker}(G_c)=\left\{h\in T_c\on{Imm}(S^1,\R^2): h=a,\, a\in\R^2\right\}\,. 
\end{equation*}
\end{lemma}
\begin{proof}
We have $h \in \on{ker}(G_c)$ if and only if $D_s h = 0$.
\end{proof}
As an immediate consequence of  Lem.~\ref{lem:kernel:metric3} we obtain that $G_c$ is a weak 
Riemannian metric on $\on{Imm}(S^1,\R^2)/\on{Tra}$. We will use Lem.~ 
\ref{lem:geod_equa_general_metric} to calculate its geodesic equation.  
\begin{theorem}
On the manifold $\on{Imm}(S^1,\R^2)/\on{Tra}$ of plane  curves modulo translations $G_c(.,.)$ 
defines a weak Riemannian metric. The geodesic equation is given by 
\begin{align*}
p &= Lc_t\otimes \ud s=\ \left(-D_s^2 c_t + D_s^2\left(\langle D_s^2 c_t, n \rangle n\right)\right)\otimes \ud s\,, 
\\
p_t &= D_s\Big(\frac12|D_sh|^2v- D_s\big( \langle D_s c_t, n \rangle\langle D_s^2 c_t, n \rangle\big) v 
\\&\phantom{=}\;\qquad\qquad\qquad +
\langle D_s^2 c_t, v \rangle\langle D_s^2 c_t, n \rangle n + \frac32\langle D_s^2 c_t, n \rangle^2 v \Big)\otimes \ud s,
\end{align*}
with the additional constraint:
\begin{align*}
c_t\in T_c\left( \on{Imm}(S^1,\R^2)/\on{Tra}\right)\cong\{h\in C^{\infty}(S^1,\R^2): h(0)=0\}\,.
\end{align*}
%\begin{align}
%\int_{S^1}c\ud s&=0, &\int_{S^1}c_t+\langle D_sc_t,v\rangle c\ud s=0\,,\label{metric3_tra_constr1}
%\end{align}
\end{theorem}
\begin{remark*}
Similarily, as in Sects. \ref{Sec:metric1} and  \ref{Sec:metric2}
the operators 
$$L_c: T_c\left(\on{Imm}(S^1,\R^2)/\on{Tra}\right)\to T_c\left(\on{Imm}(S^1,\R^2)/\on{Tra}\right)$$ 
are not elliptic operators and 
thus we again cannot apply the well-posedness results from \cite{Michor2007} or \cite{Bauer2011b}. 
\end{remark*}

\begin{proof}
Using the variational formulas from Sect.~\ref{var} we calculate
the variation of the metric:
\begin{align*}
D_{c,m} (G_c(h,h))&= \int_{S^1} -\langle D_sm,v\rangle\langle D_sh,D_sh\rangle 
- 2\langle D_s^2 h, n \rangle\langle D_s h, n \rangle D_s \left(\langle D_s m, v \rangle \right) 
\\&\qquad\qquad
- 2\langle D_s^2 h, n \rangle\langle D_s^2 h, v \rangle \langle D_s m, n \rangle  
- \langle D_s^2 h, n \rangle^2\langle D_s m, v \rangle
\ud s\,.
\end{align*}
We can now calculate  $H_c(h,h)$ using a series of integrations by parts,
\begin{align*}
H_c(h,h)&= D_s\Big(|D_sh|^2v- 2 D_s\big( \langle D_s h, n \rangle\langle D_s^2 h, n \rangle\big) v 
\\&\phantom{=}\;\qquad 
+2   \langle D_s^2 h, v \rangle\langle D_s^2 h, n \rangle n + 3\langle D_s^2 h, n \rangle^2 v \Big)\otimes \ud s\,.
\end{align*}
Regarding the existence of the geodesic equation, see the proof of Thm.~\ref{thm:m1_geodesic_equation}. 
\end{proof}

\subsection{The $R$-transform}
Consider the map
\begin{equation*}\label{Rtrans_metric3}
 R: \left\{
\begin{array}{ccc}
\on{Imm}(S^1, \R^2)/\on{Tra}&\to& C^\infty(S^1, \R_{>0} \x S^1\x\R) \,, \\
c&\mapsto&(\sqrt{|c'|}, \al, \ka|c'|^2)
\end{array} \right.\,.
\end{equation*}

In order to simplify the notation we shall write, $\R^3_\ast$ for $\R_{>0} \x S^1\x\R$. On 
$\R^3_\ast$ we define the following Riemannian metric 
\begin{equation}
\label{metric3_g}
g_{(q_1,q_2,q_3)} = \left( \begin{array}{ccc}
4 & 0 & 0 \\
0 & q_1^2 & 0 \\
0 & 0 & q_1^{-6}
\end{array} \right)\,,
\end{equation}
and we equip the space $C^\infty(M, (\R^3_\ast, g))$ with the $L^2$-Riemannian metric,
\begin{equation}
G_q^{L^2}(h, k) = \int_M g_{q(\th)}(h(\th), k(\th)) \ud \th\,.
\end{equation}
Here $q \in C^\infty(M, (\R^3_\ast, g))$ and $h, k \in T_q C^\infty(M, (\R^3_\ast, g))$. 

The reason for introducing these objects lies in the following theorem.
\begin{theorem}
\label{thm:m3_rmap}
With the metrics $g$ and $G^{L^2}$ defined as above, the $R$-transform
\[
R\,:\, (\on{Imm}(M, \R^2)/\on{Tra}, G) \to (C^\infty(M, (\R^3_\ast, g)), G^{L^2})
\]
is an injective isometry between weak Riemannian manifolds, i.e., 
\[
G_c(h, k) = G_{R(c)}^{L^2}(dR(c).h, dR(c).k)=\int_{M} g_{R(c)(\th)}(dR(c).h(\th), dR(c).k(\th)) \ud \th\,,
\]
for  $h,k \in T_c \on{Imm}(M,\R^2)/\on{Tra}$.
%For $q=(q_1,q_2,q_3)\in\on{im}(R)$ an inversion formula is given via
%$$R\i(q)\begin{cases}
%\on{im}(R)&\to \on{Imm}(M, \R^2)/\on{Tra}  \\  
%q &\to \int_0^\th q_1^2e^{iq_2}\ud \th
%        \end{cases}
%$$
%Thus $R$ is in particular injective. 
\end{theorem}
\begin{proof}
Using the formulas from Sect.~\ref{var}
we calculate the derivative of the $R$-transform:
\[
dR(c)h = \left( \tfrac 12 \langle D_s h, v \rangle \sqrt{|c_\th|}, 
\langle D_s h, n \rangle,
\langle D_s^2 h, n \rangle |c_\th|^2 \right)\,.
\]
Hence
\[
g_{R(c)}(dR(c)h,dR(c)h) = \left(\langle D_c h, v \rangle^2 + \langle D_s h, n \rangle^2 
+ \langle D_s^2 h, n \rangle^2 \right) |c_\th|\,,
\]
and the first statement of the theorem follows. The map $R$ is injective on $\on{Imm}/\on{Tra}$ 
since one can reconstruct the curve $c$ up to translations from $|c'|$ and $\al$. 
\end{proof}

\begin{remark*}
A key difference between this $R$-transform and the ones used in Sects. \ref{Sec:metric1:Rtrans} 
and \ref{Sec:metric2:Rtrans} is that the image has infinite codimension, both for open as well as 
closed curves.  
\end{remark*}

\subsection{The space of closed curves}

 \begin{theorem}
\label{thm:m3_im_closed}
The image of $\on{Imm}(S^1,\R^2)/{\on{Tra}}$ under the $R$-transform is given by
\[
\on{im}(R) = \left\{ (q_1,q_2,q_3) \,:\,
\begin{array}{l}
(1)\, q_2' = q_1^{-2}q_3 \\
(2)\, \int_{S^1} q_1^2 \exp(i q_2) \ud \th=0
\end{array} \right\} \,.
\]
It is a splitting submanifold of $C^\infty(S^1,\R^3_{\ast})$. The inverse of the $R$-transform is given by
\[
R\i:\left\{ 
\begin{array}{ccc}
\on{im}(R) &\to & \on{Imm}(S^1,\R^2)/{\on{Tra}} \\
q &\mapsto & \int_0^\th q_1^2 \exp(i q_2) \ud \si
\end{array} \right. \,.
\]
\end{theorem}
We introduce the maps
\begin{align}
\label{eq:m3_constraints}
H_{\on{diff}} &: \left\{ 
\setlength{\arraycolsep}{0.3\arraycolsep}
\begin{array}{ccc}
C^\infty(S^1,\R^3_\ast) &\to &C^\infty(S^1,\R) \\
q &\mapsto & q_3 - q_1^2q_2'
\end{array} \right. &
H_{\on{cl}} &: \left\{ 
\setlength{\arraycolsep}{0.3\arraycolsep}
\begin{array}{ccc}
C^\infty(S^1,\R^3_\ast) &\to &\R^2 \\
q &\mapsto & \int_{S^1} q_1^2 \exp(iq_2) \ud \th
\end{array} \right. \,,
\end{align}
which allow us to write $\on{im}(R) = H_{\on{diff}}\i(0) \cap H_{\on{cl}}\i(0)$.

\begin{proof}
The constraint $H_{\on{diff}}(q)=0$, rewritten in terms of $c$, is the definition of curvature 
$\ka = D_s\al$, written as 
\[
\ka|c'|^2 = \sqrt{|c'|}^2 \al'\,.
\]
The constraint $H_{\on{cl}}(0)$ results from $\int_{S^1} c' \ud \th = 0$ and $c' = q_1^2 
\exp(iq_2)$. The latter identity implies 
\[
c(\th) = c(0) + \int_0^\th q_1^2 \exp(i q_2) \ud \si\,.
\]
Since $c(0)$ is unspecified, this determines the curve up to translations. Thus we have identified 
$\on{im}(R)$ and proven the inversion formula. 

To show that $\on{im}(R)$ is a splitting submanifold define
\[
\Ph : \left\{ \begin{array}{ccc}
C^\infty(S^1,\R^3_\ast) &\to & C^\infty(S^1,\R^3_\ast) \\
(q_1,q_2,q_3) &\mapsto &(q_1,q_2,q_3 + q_1^2q_2')
\end{array} \right.\,.
\]
The map $\Ph$ is a smooth diffeomorphism and $H_{\on{diff}}\i(0) = \Ph(\{ q_3 = 0\})$. 
Thus $H_{\on{diff}}\i(0)$ is a splitting submanifold of $C^\infty(S^1,\R^3_\ast)$. 
Now we will pull back $\on{im}(R)$ by $\Ph$ and show that $\Ph\i(\on{im}(R))$ is a splitting 
submanifold of $\{q_3=0\}$. First note that 
\begin{align*}
\Ph\i(\on{im}(R)) = \left(H_{\on{cl}}\o \Ph\right)\i(0) \cap \{ q_3 = 0 \}
= H_{\on{cl}}\i(0) \cap \{ q_3 = 0 \}\,.
\end{align*}
The last equality holds because $H_{\on{cl}}$ doesn't depend on $q_3$. That $H_{\on{cl}}\i(0)$ is a 
splitting submanifold of $\{ q_3 = 0\}$ can be shown via the implicit function theorem with 
convenient parameters like in Lem.~\ref{lem:m1_imcl_submf}. This concludes the proof.  
\end{proof}

Next we want to compute the orthogonal projection 
$\on{Proj}^{\on{im}} : TC^\infty(S^1,\R^3_\ast)\upharpoonright \on{im}(R) \to T\on{im}(R)$.
We do this in two steps. First define the space 
\[
\on{im}(R)_{\on{op}} = H_{\on{diff}}\i(0) = \{ q \in C^\infty(S^1,\R_{>0}\x S^1\x \R) : q_3 = q_1^{2}q_2' \}\,,
\]
and compute the orthogonal projection 
$\on{Proj}^{\on{im}}_{op}: T C^\infty(S^1,\R^3_\ast)\to T\on{im}(R)_{op}$. 
The space $\on{im}(R)_{\on{op}}$ corresponds to the image of the $R$-transform on the space of open 
curves, if $S^1$ were replaced by $[0,2\pi]$. We can compute the tangent space 
\[
T_q \on{im}(R)_{\on{op}} = \left\{ \left(h_1,h_2,2q_1^{-1} q_3h_1 + q_1^2 h_2' \right)\,:\, h_1,h_2 \in C^\infty(S^1,\R) \right\}
\]
for $q \in \on{im}(R)_{\on{op}}$.

\begin{lemma}
\label{lem:m2_orth_proj_op}
Let $q \in \on{im}(R)_{\on{op}}$ and $h \in T_q C^\infty(S^1,\R^3_\ast)$. 
Then the orthogonal projection $k = \on{Proj}^{\on{im}}_{\on{op}}(q).h$ of $h$ to the tangent space 
to $\on{im}(R)_{\on{op}}$ is given by 
\begin{subequations} \label{eq:m3_orth_proj_op}
\begin{alignat}{1}
q_1^2 k_2- \left( A k_2'\right)' &= B \\
k_1 &= \frac{2q_1^8h_1 + q_1q_3h_3 - q_1^3q_3k_2'}{2(q_1^8+q_3^2)}  \\
k_3 &= 2q_1q_2'k_1 + q_1^2k_3'
\end{alignat}
\end{subequations}
with the functions
\begin{align*}
A&= \frac{q_1^6}{q_1^8 + q_3^2}\,, &
B&= q_1^2h_2 + \left( \frac{2q_1^3q_3h_1 + q_1^{-4}q_3^2h_3}{q_1^8 + q_3^2} \right)'\,,
\end{align*}
Furthermore the map $\on{Proj}^{\on{im}}_{\on{op}}$ is smooth.
\end{lemma}

\begin{proof}
The orthogonal projection $k \in T_q\on{im}(R)_{\on{op}}$ of $h$ is defined by the equation
\[
G_q^{L^2}(h, a) = G_q^{L^2}(k, a)\,,
\]
which has to hold for all $a \in T_q\on{im}(R)_{\on{op}}$. Any such $a$ is of the form
\[
a = (a_1, a_2, 2q_1^{-1}q_3a_1 + q_1^2a_2')\,.
\] 
Thus the above equation reads
\begin{multline*}
\int_{S^1} 4h_1a_1 + q_1^2h_2 a_2 + q_1^{-6}h_3(2q_1^{-1}q_3a_1 + q_1^2a_2') \ud \th ={} \\
{}= \int_{S^1} 4k_1a_1 + q_1^2k_2 a_2 + q_1^{-6}(2q_1^{-1}q_3k_1 + q_1^2k_2')(2q_1^{-1}q_3a_1 + q_1^2a_2') \ud \th\,.
\end{multline*}
which is equivalent to the system
\begin{align*}
4h_1 + 2q_1^{-7}q_3h_3 &=
4k_1 + 4q_1^{-8}q_3^2k_1 + 2q_1^{-5}q_3k_2' \\
q_1^2 h_2 - (q_1^{-4}h_3)' &=
q_1^2 k_2 - \left(2q_1^{-5}q_3k_1 + q_1^{-2}k_2'\right)'\,.
\end{align*}
The first equation allows us to express $k_1$ in terms of $k_2'$,
\[
k_1 = \frac{2q_1^8h_1 + q_1q_3h_3 - q_1^3q_3k_2'}{2(q_1^8+q_3^2)}\,,
\]
which we then insert into the second equation and obtain
\[
q_1^2 k_2- \left( A k_2'\right)' = B\,,
\]
with $A$ and $B$ defined as above. The map $(q, h) \mapsto k$ is smooth, because the solution of an elliptic equation depends smoothly on the coefficients.
\end{proof}

Now we compute $\on{Proj}^{\on{im}}$. The components of the constraint function $H_{\on{cl}}$ are
\[
H_{\on{cl}} = \left( \int_{S^1} q_1^2 \cos q_2 \ud \th, \int_{S^1} q_1^2 \sin q_2 \ud \th \right)\,,
\]
and their gradients with respect to the $G^{L^2}$-metric are given by
\begin{subequations} \label{eq:m3_hcl_grad}
\begin{alignat}{1}
\on{grad}^{L^2} H_{\on{cl}}^1(q) &= (\tfrac 12 q_1 \cos q_2, -\sin q_2, 0) \\
\on{grad}^{L^2} H_{\on{cl}}^2(q) &= (\tfrac 12 q_1 \sin q_2, \cos q_2, 0)\,.
\end{alignat}
\end{subequations}

\begin{theorem}
\label{thm:m3_orth_proj}
Let $q \in \on{im}(R)$. 
Define $v^i = \on{Proj}^{\on{im}}_{\on{op}}(q).\on{grad}^{L^2} H_{\on{cl}}^i(q)$ for $i=1,2$ 
and let $T(q) : T_q C^\infty(S^1,\R^3_\ast) \to T_q C^\infty(S^1,\R^3_\ast)$ 
be the projection to the orthogonal complement of $\on{span}\{v^1, v^2\}$. Then
\begin{equation}
\label{eq:m3_orth_proj}
\on{Proj}^{\on{im}}(q) = T(q) \o \on{Proj}^{\on{im}}_{\on{op}}(q)\,,
\end{equation}
and $\on{Proj}^{\on{im}}$ is smooth.
\end{theorem}

\begin{proof}
Since $\on{span}\{v^1,v^2\}$ is the orthogonal complement to $\on{im}(R)$ within $\on{im}(R)_{op}$, 
the map $T(q)$ projects $T_q \on{im}(R)_{op}$ down to $T_q \on{im}(R)$. In order for $T$ to be 
smooth we need that $\{v^1,v^2\}$ are linearly independent. This is equivalent to the condition 
that $\on{grad}^{L^2} H_{\on{cl}}^1(q)$ and $\on{grad}^{L^2} H_{\on{cl}}^2(q)$ don't differ by an 
element of $\left(T_q \on{im}(R)_{op}\right)^\perp$. This is clear from Lem.~
\ref{lem:m2_orth_proj_op} and the formulas \eqref{eq:m3_hcl_grad} for the gradients of 
$H_{\on{cl}}^i$.      
\end{proof}

The geodesic equation on $\on{Imm}_{\on{conv}}(S^1,\R^2)/\on{Tra}$ is well-posed in appropriate 
Sobolev completions. We refer to Sect.~\ref{Sec:wp} and in particular to Sect.~\ref{Sec:wp_1_and_2} 
for a detailed discussion and proofs. The spaces $\on{Imm}^{j,k}(S^1,\R^2)$ are defined in 
\eqref{eq:custom_sob_scale}.   

\begin{theorem}
For $k\geq 2$ and initial conditions $(c_0, u_0) \in T\on{Imm}^{k+1,k+2}(S^1,\R^2)$, the 
geodesic equation has solutions in $\on{Imm}^{j+1,j+2}(S^1,\R^2)$ for each $2 \leq j \leq k$. The 
solutions depend $C^\infty$ on $t$ and the initial conditions and the domain of existence is 
independent of $j$.   

In particular, for smooth initial conditions $(c_0, u_0) \in T\on{Imm}(S^1,\R^2)$ the geodesic 
equation has smooth solutions. 
\end{theorem}

\section{The full $H^2$-metric}
\label{Sec:metric4}

\subsection{The metric and its geodesic equation}
In this section we will study the  full $H^2$-metric that has  translations in the kernel. In 
comparison to the metric of the previous section we add  the missing $H^2$-term 
$\langle D^2_sh,v\rangle$ to the definition of the metric. This yields the bilinear form   
\begin{equation}
\label{metric4}
G_c(h,k) = \int_M \langle D_s h, D_s k \rangle  + \langle D_s^2 h, D_s^2 k\rangle \ud s\,.
\end{equation}
On closed curves the associated operator field of this pseudo metric is given by
\[
L_c h = D_s^4 h - D_s^2 h\,.
\]
% \begin{equation*}\label{inertia_operator_metric4}
%  L_c: \left\{
% \begin{array}{ccc}
% T_c\on{Imm}(S^1,\R^2)&\to& T_c\on{Imm}(S^1,\R^2)\,, \\
% h&\mapsto&D_s^4 h  -D^2_sh
% \end{array} \right.\,.
% \end{equation*}
\begin{lemma}\label{lem:kernel:metric4}
The null space of the bilinear form $G_c(.,.)$ is spanned by constant vector fields, i.e.,
\begin{equation*}
 \on{ker}(G_c)=\left\{h\in T_c\on{Imm}(M,\R^2): h=a,\, a\in\R^2\right\}\,. 
\end{equation*}
\end{lemma}
\begin{proof}
The proof of this lemma is obvious, since the bilinear form $G_c$ measures the full first derivative.
\end{proof}
\begin{remark*}
In contrast to the other metrics studied in this article the operator $L$ is elliptic and 
invertible on $T_c\on{Imm}/\on{Tra}$. Therefore we will be able to apply the well-posedness result 
of \cite{Bauer2011b,Michor2008a}.  
\end{remark*}

As an immediate consequence of  Lem. \ref{lem:kernel:metric1} we obtain that $G_c$ is a weak 
Riemannian metric on $\on{Imm}(M,\R^2)/\on{Tra}$.  
Similarly. as in Sects. \ref{Sec:metric1} and \ref{Sec:metric2} we will use  Lem. 
\ref{lem:geod_equa_general_metric} to calculate its  geodesic equation.
\begin{theorem}
On the manifold $\on{Imm}(M,\R^2)/\on{Tra}$ of plane  curves modulo translations,  the pseudo metric 
$G_c(.,.)$ is a weak Riemannian metric.  
The geodesic equation on the manifold of closed curves modulo Euclidean motions 
$\on{Imm}(S^1,\R^2)/\on{Tra}$ is given by: 
\begin{align*}
p &=  L c_t\otimes \ud s=D_s^4 c_t  -D^2_s c_t\otimes \ud s
\\
p_t &= D_s\Big(3|D^2_sc_t|^2v- 2 D_s\big( \langle D_s c_t, D^2_sc_t \rangle v\big)+|D_sc_t|^2v\Big) \otimes\ud s\end{align*}
with the additional constraint:
\begin{align*}
c_t\in T_c\left( \on{Imm}(M,\R^2)/\on{Tra}\right)\cong\{h\in C^{\infty}(M,\R^2): h(0)=0\}\,.
\end{align*}
For each  $k> 11/2$ and initial conditions $(c_0, u_0) \in T\on{Imm}^{k}(S^1,\R^2)$, 
the geodesic equation has solutions in $\on{Imm}^{j}(S^1,\R^2)$ for each $11/2 < j \leq k$. The solutions depend $C^\infty$ on 
$t$ and the initial conditions and the domain of existence is independent of $j$. 
In particular, for smooth initial conditions $(c_0, u_0) \in T\on{Imm}(S^1,\R^2)$, the geodesic 
equation has smooth solutions. 
\end{theorem}

\begin{proof}
Using the variational formulas from Sect.~\ref{var} and 
\begin{align*}
d (D_s^2h)(c).m =-2\langle D_sm,v\rangle D_s^2h-D_s(\langle D_sm,v\rangle)D_sh
\end{align*}
we obtain the variation of the metric:
\begin{align*}
D_{c,m} (G_c(h,h))&= \int_{S^1}-3\langle D_sm,v\rangle \langle D_s^2h,D_s^2h\rangle
-2D_s(\langle D_sm,v\rangle) \langle D_sh,D_s^2h\rangle
\\&\qquad\qquad
- \langle D_sm,v\rangle\langle D_sh,D_sh\rangle\ud s
\end{align*}
We can now calculate  $H_c(h,h)$ using a series of integration by parts:
\begin{align*}
H_c(h,h)&= D_s\Big(3|D^2_sh|^2v- 2 D_s\big( \langle D_s h, D^2_sh \rangle v\big)+|D_sh|^2v\Big) \otimes\ud s\,.\qedhere
\end{align*}
The well-posedness result can be proven similar as in \cite{Michor2008a}, see also \cite{Bauer2011b}.
\end{proof}

\subsection{The $R$-transform}
We introduce the following transformation
\begin{equation*}\label{Rtrans_metric4}
 R :\left\{
\begin{array}{ccc}
\on{Imm}(M, \R^2)/\on{Tra}&\to& C^\infty(M, \R_{>0}\x S^1\x\R^2) \,, \\
c&\mapsto&(\sqrt{|c'|}, \al,D_s |c'|, \ka|c'|^2)
\end{array} \right.\,.
\end{equation*}
which assigns to each curve $c$ the 4-tuple of functions $(\sqrt{|c'|}, \al, D_s|c'|, \ka|c'|^2)$. 
In order to simplify the notation we shall write, $\R^4_\ast$ for $\R_{>0}\x S^1\x\R^2$. On $\R^4_\ast$ 
we define the following Riemannian metric,
\begin{equation}
\label{metric4_g}
g_{(q_1,q_2,q_3,q_4)} = \left( \begin{array}{cccc}
4 & 0 & 0&0 \\
0 & q_1^2+q_4^2q_1^{-6} & -q_4q_1^{-4}&0 \\
0&-q_4q_1^{-4}&q_1^{-2}&0\\
0 & 0 & 0&q_1^{-6}
\end{array} \right)\,,
\end{equation}
and we equip the space $C^\infty(M, (\R^4_\ast, g))$ with the $L^2$-Riemannian metric,
\begin{equation}
\label{l2_met_r4}
G_q^{L^2}(h, k) = \int_M g_{q(\th)}(h(\th), k(\th)) \ud \th\,.
\end{equation}
Here $q \in C^\infty(M, (\R^4_\ast, g))$ and $h, k \in T_q C^\infty(M, (\R^4_\ast, g))$. 

The reason for introducing these objects lies in the following theorem:
\begin{theorem}
\label{thm:m4_rmap}
With the metrics $g$ and $G^{L^2}$ defined as above, the map
\[
R\,:\, (\on{Imm}(M, \R^2)/\on{Tra}, G) \to (C^\infty(M, (\R^4_\ast, g)), G^{L^2})
\]
is an injective isometry between weak Riemannian manifolds, i.e., 
\[
G_c(h, k) = G_{R(c)}^{L^2}(dR(c).h, dR(c).k)=\int_{M} g_{R(c)(\th)}(dR(c).h(\th), dR(c).k(\th)) \ud \th\,,
\]
for  $h,k \in T_c \on{Imm}(M,\R^2)/\on{Tra}$.
\end{theorem}
\begin{proof}
Using  the identity
\begin{align*}
D_{c,h} \left(D_s \o R\right) &= D_s \left(D_{c,h} R\right) - \langle D_s h, v \rangle D_s (R(c))\,,
%D_{c,h} \left(\sqrt{|c'|}\right) &= \frac12 \sqrt{|c'|}\langle D_sh,v\rangle
%D_{c,h} \left(|c'|\right) &=|c'|\langle D_sh,v\rangle
\end{align*}
we calculate:
\begin{align*}
D_{c,h}\left(D_s |c'|\right)&=  D_s \left(|c'| \langle D_sh,v\rangle\right) -  \langle D_s h, v \rangle D_s |c'|\\&
=|c'|\langle D_s^2h,v\rangle+|c'|\langle D_s^2h,D_sv\rangle\\
&=|c'|\langle D_s^2h,v\rangle+|c'|\ka \langle D_sh,n\rangle\,.
\end{align*} 
Thus we obtain  the derivative of the $R$-map:
\[
dR(c).h = \left( \tfrac 12 \langle D_s h, v \rangle \sqrt{|c'|}, 
\langle D_s h, n \rangle, |c'|\langle D_s^2h,v\rangle+|c'|\ka \langle D_sh,n\rangle,
\langle D_s^2 h, n \rangle |c'|^2 \right)\,.
\]
Hence
\begin{align*}
g_{R(c)}(dR(c).h,dR(c).h)& = \Big(\langle D_s h, v \rangle^2 + \langle D_s h, n \rangle^2(1+\ka^2) 
\\&\qquad
+\langle D_s^2 h, v \rangle^2+ 2\langle D_s^2 h, v \rangle \langle D_s h, n \rangle\ka
+   \langle D_s h, n \rangle^2\ka^2
\\&\qquad
-2 \langle D_s^2 h, v \rangle \langle D_s h, n \rangle\ka-2\langle D_s h, n \rangle^2\ka^2
+\langle D_s^2 h, n \rangle^2 \Big) |c'|
\\&
=\left(\langle D_s h, D_sh \rangle+\langle D_s^2 h, D_s^2h \rangle\right)|c'|\,,
\end{align*}
and the first statement of the theorem follows. The map $R$ is injective on $\on{Imm}/\on{Tra}$ 
since one can reconstruct the curve $c$ up to translations from $|c'|$ and $\al$. 
\end{proof}

\subsection{The metric on the space of closed curves}

\begin{theorem}
The image of $\on{Imm}(S^1,\R^2)/{\on{Tra}}$ under the $R$-transform is given by
\[
\on{im}(R) = \left\{ (q_1,q_2,q_3,q_4) \,:\,
\begin{array}{l}
(1)\,2q_1'=q_1q_3 \\
(2)\, q_2' = q_1^{-2}q_4\\
(3)\,  \int_{S^1} q_1^2 \exp(i q_2) \ud \th=0
\end{array} \right\} \,.
\]
It is a splitting submanifold of $C^\infty(S^1,\R^4_{\ast})$. The inverse of the $R$-transform is given by
\[
R\i:\left\{ 
\begin{array}{ccc}
\on{im}(R) &\to & \on{Imm}(S^1,\R^2)/{\on{Tra}} \\
q &\mapsto & \int_0^\th q_1^2 \exp(i q_2) \ud \si
\end{array} \right. \,.
\]
\end{theorem}
Introduce the maps
\begin{align*}
H_{\on{diff}} &: \left\{ 
\setlength{\arraycolsep}{0.3\arraycolsep}
\begin{array}{ccc}
C^\infty(S^1,\R^4_\ast) &\to &C^\infty(S^1,\R^2) \\
q &\mapsto & \left(q_3-2q_1^{-1}q_1',q_4 - q_1^2q_2'\right)
\end{array} \right. \\
H_{\on{cl}} &: \left\{ 
\setlength{\arraycolsep}{0.3\arraycolsep}
\begin{array}{ccc}
C^\infty(S^1,\R^3_\ast) &\to &\R^2 \\
q &\mapsto & \int_{S^1} q_1^2 \exp(iq_2) \ud \th
\end{array} \right. \,,
\end{align*}
which allow us to write $\on{im}(R) = H_{\on{diff}}\i(0) \cap H_{\on{cl}}\i(0)$.

\begin{proof}
The first component of the constraint $H_{\on{diff}}(q)=0$ corresponds to the identity 
$D_s|c'|=\frac{1}{|c'|}\p_\th |c'|$. The second component as well as $H_{\on{cl}}(q)=0$ are the 
same as in Thm. \ref{thm:m3_im_closed}.  

To show that $\on{im}(R)$ is a splitting submanifold define
\[
\Ph : \left\{ \begin{array}{cccc}
C^\infty(S^1,\R^4_\ast) &\to & C^\infty(S^1,\R^4_\ast) \\
(q_1,q_2,q_3,q_4) &\mapsto &(q_1,q_2,q_3+2q_1^{-1}q_1',q_4 + q_1^2q_2')
\end{array} \right.\,.
\]
We note as in the proof of Thm. \ref{thm:m3_im_closed} that $\Ph$ is a smooth diffeomorphism, that 
$H_{\on{diff}}\i(0) = \Ph(\{ q_3 = 0, q_4=0\})$ and that 
\[
\Ph\i(\on{im}(R))
= H_{\on{cl}}\i(0) \cap \{ q_3 = 0,q_4=0 \}\,.
\]
Since $H_{\on{cl}}$ doesn't depend on $q_3, q_4$ we then apply the implicit function theorem with 
convenient parameters to conclude the proof. 
\end{proof}

\section{Well-posedness of the geodesic equation}
\label{Sec:wp}

\subsection{Well-posedness for the third metric}
Let $G$ be the Riemannian metric \eqref{metric3} from Sect. \ref{Sec:metric3} on closed curves
\[
G_c(h,h) = \int_{S^1} |D_s h|^2 + \langle D_s^2 h, n\rangle^2 \ud s\,.
\]
For $j\geq 1$, $k \geq 1$ define the spaces
\begin{equation}
\label{eq:custom_sob_scale}
\on{Imm}^{j,k}(S^1,\R^2) = \left\{ c \in H^2(S^1,\R^2) \,:\, |c'| \in H^j(S^1,\R), \al \in H^k(S^1,S^1) \right\}\,.
\end{equation}
Then $\on{Imm}^{j,k}(S^1,\R^2)$ is a Hilbert manifold modelled on $H^j \x H^k$ and a global chart 
is given by $c \mapsto (|c'|, \al)$. For $k\geq 2$ denote by  
\[
\on{Imm}^k(S^1,\R^2) = \{ c \in H^k(S^1,\R^2) \,:\, |c'| > 0 \}
\]
the space of Sobolev immersions of order $k$. Note that we have the inclusions
\[
\on{Imm}^{\max(j,k)+1}(S^1,\R^2) \subseteq \on{Imm}^{j,k}(S^1,\R^2) \subseteq \on{Imm}^{\min(j,k)+1}(S^1,\R^2)\,,
\]
and if $j=k$, then
\[
\on{Imm}^{j,j}(S^1,\R^2) = \on{Imm}^{j+1}(S^1,\R^2)\,.
\]
The spaces can $\on{Imm}^{j,k}(S^1,\R^2)$ can be seen as a refinement of the Sobolev scale of function spaces.

\begin{theorem}
\label{thm:m3_spray_smooth}
For $k\geq 2$, the geodesic spray 
\[
\Xi^G : T\on{Imm}^{k+1,k+2}(S^1,\R^2)/\on{Tra} \to TT\on{Imm}^{k+1,k+2}(S^1,\R^2)/\on{Tra}
\]
of the metric $G$ is smooth.
\end{theorem}

Combining this theorem with the existence theorem for ODEs, the translation invariance of the 
geodesic spray from App. \ref{Sec:trans_inv} and Thm.~\ref{thm:max_domain}, we obtain the following 
corollary.  

\begin{corollary}
\label{cor:m3_solutions_smooth}
For $k\geq 2$ and initial conditions
$
(c_0, u_0) \in T\on{Imm}^{k+1,k+2}(S^1,\R^2)\,,
$ the 
geodesic equation has solutions in $\on{Imm}^{j+1,j+2}(S^1,\R^2)/\on{Tra}$ for each $2 \leq j \leq k$. The 
solutions depend $C^\infty$ on $t$ and the initial conditions and the domain of existence is 
independent of $j$.   

In particular for smooth initial conditions $(c_0, u_0) \in T\on{Imm}(S^1,\R^2)/\on{Tra}$ the geodesic 
equation has smooth solutions. 
\end{corollary}

\begin{proof}[Proof of Theorem]
The $R$-transform
\[
R(c)=(\sqrt{|c'|}, \al, \ka|c'|^2)
\]
extends to a smooth map
\[
R:\on{Imm}^{k+1,k+2}(S^1,\R^2)/\on{Tra} \to H^k(S^1,\R^3_{\ast})
\]
with $\R^3_\ast = \R_{>0}\x S^1\x\R$ and the image of the map is given by
\[
\on{im}(R) = \left\{ (q_1,q_2,q_3) \,:\,
\begin{array}{l}
(1)\, q_1 \in H^k, q_2 \in H^{k+1}, q_3 \in H^k \\
(2)\, q_2' = q_1^{-2}q_3 \\
(3)\, \int_{S^1} q_1^2 \exp(i q_2) \ud \th=0
\end{array} \right\} \,.
\]
By Lem.~\ref{lem:sob_im_r} it is an embedded submanifold of $H^k(S^1,\R^3_\ast)$. Let $g$ be the 
Riemannian metric \eqref{metric3_g} on $\R^3_\ast$ and $G^{L^2}$ the $L^2$-metric~\eqref{l2_met_r3} 
on $H^{k}(S^1,\R^3_\ast)$. The same proof as for Thm.~\ref{thm:m3_rmap} shows that the 
$R$-transform is an isometry between $(\on{Imm}^{k+1,k+2}(S^1,\R^2)/\on{Tra}, G)$ and 
$(H^{k}(S^1,\R^3_\ast), G^{L^2})$.    

Denote by
\[
\io : \on{im}(R) \to H^{k}(S^1,\R^3_\ast)
\]
the inclusion map. The orthogonal projection to $T\on{im}(R)$ is a map
\[
\on{Proj}^{\on{im}} : \io^\ast TH^k(S^1,\R^3_\ast)\upharpoonright \on{im}(R) \to T \on{im}(R)\,,
\]
and it is given by the same formulas \eqref{eq:m3_orth_proj_op} and \eqref{eq:m3_orth_proj} as in 
Thm.~\ref{thm:m3_orth_proj}. It is shown in Lem.~\ref{lem:m3_proj_hk} that $\on{Proj}^{\on{im}}$ 
can be extended to a smooth map  
\[
\on{Proj}^{\on{im}} : TH^k(S^1,\R^3_\ast) \to T H^k(S^1,\R^3_\ast)\,.
\]
Denote by
\[
\Xi^{L^2} : TH^k (S^1,\R^3_\ast) \to TTH^k(S^1,\R^3_\ast)
\]
the geodesic spray of the $G^{L^2}$-metric on $H^k(S^1,\R^3_\ast)$. It is smooth by Thm.~
\ref{thm:ebin1970}. Theorem \ref{thm:ebin1970_proj} shows that the geodesic spray of the 
$G^{L^2}$-metric restricted to $\on{im}(R)$ is given by  
\[
\Xi^{\on{im}} = T\on{Proj}^{\on{im}} \o \Xi^{L^2} \o T\io : T\on{im}(R) \to TT\on{im}(R)\,,
\]
and that this map is smooth as well.
% To see that this composition is well-defined and indeed maps into $TT\on{im}(R)$ we need to verify that for $X \in T\on{im}(R)$ we have $\Xi^{L^2} \o T\io(X) \in T \left(\io^\ast TH^k(S^1,\R^3_\ast)\right)$. This follows from
% \begin{align*}
% \Xi^{L^2} \o T\io(X) \in T \left(\io^\ast TC^\infty(S^1,\R^3_\ast)\right) 
% &\,\on{iff}\, \pi_{TTC^\infty} \o \Xi^{L^2}\o T\io(X) \in \io^\ast TC^\infty(S^1,\R^3_\ast) \\
% &\,\on{iff}\, \pi_{TC^\infty} \o \pi_{TTC^\infty} \o \Xi^{L^2}\o T\io(X) \in \on{im}(R) \\
% &\,\on{iff}\, \pi_{TC^\infty} \o \underbrace{T\pi_{TC^\infty} \o \Xi^{L^2}}_{=\on{Id}}\o T\io(X) \in \on{im}(R) \\
% &\,\on{iff}\, \pi_{TC^\infty} \o  T\io(X) \in \on{im}(R) \\
% &\,\on{iff}\, \io \o \pi_{T\on{im}}(X) \in \on{im}(R)\,.
% \end{align*}
% In particular we see that $\Xi^{\on{im}}$ is smooth. 
The geodesic sprays $\Xi^{L^2}$ and $\Xi^{\on{im}}$ are $TR$-related, i.e., the following diagram 
commutes. 
\[
\xymatrix{
TT\on{Imm}^{k+1,k+2}/\on{Tra} \ar[r]^-{TTR} & TT\on{im}(R) \\
T\on{Imm}^{k+1,k+2}/\on{Tra} \ar[r]^-{TR} \ar[u]^{\Xi^{G}} & T\on{im}(R) \ar[u]_{\Xi^{\on{im}}} \\
}
\]
Since $\on{im}(R)$ is an embedded submanifold of $H^k(S^1,\R^3_{\ast})$ the map 
\[
R : \on{Imm}^{k+1,k+2}(S^1,\R^3_\ast)/\on{Tra} \to \on{im}(R)
\]
is a diffeomorphism and we can conclude that $\Xi^G$ is smooth.
\end{proof}

\begin{lemma}
\label{lem:sob_im_r}
Let $k\geq 2$. The image of the $R$-transform
\[
R: \left\{
\begin{array}{ccc}
\on{Imm}^{k+1,k+2}(S^1,\R^2)/\on{Tra} &\to &H^k(S^1,\R^3_{\ast}) \\
c &\mapsto &(\sqrt{|c'|}, \al, \ka|c'|^2)
\end{array} \right.\,,
\]
is an embedded submanifold of $H^k(S^1,\R^3_{\ast})$.
\end{lemma}
\begin{proof}
Define the functions
\begin{align*}
C_\al &: H^k(S^1,\R^3_\ast) \to \R &
C_\al(q) &= \int_{S^1} q_1^{-2} q_3 \ud \th \\
C_{\on{diff}} &: H^k\x H^{k+1} \x H^k \to H^k &
C_{\on{diff}} &=  q_2' - q_1^{-2}q_3 \\
C_{\on{cl}} &: H^k(S^1,\R^3_\ast) \to \R^2 &
C_{\on{cl}}(q) &= \int_{S^1} q_1^{2} \exp(iq_2) \ud \th\,,
\end{align*}
which allow us to write the image $\on{im}(R)$ as
\[
\on{im}(R) = H^k(S^1,\R^3_\ast) \cap C_\al\i(2\pi\mathbb Z) \cap C_{\on{diff}}\i(0) \cap
C_{\on{cl}}\i(0)\,.
\]
That $H^k(S^1,\R^3_\ast) \cap C_\al\i(2\pi\mathbb Z)$ is a submanifold of $H^k(S^1,\R^3_\ast)$ can 
be seen via the inverse function theorem on Banach spaces. Let $\Ph$ be the map 
\[
\Ph:\left\{ \begin{array}{ccc}
H^k(S^1,\R^3_\ast) \cap C_\al\i(2\pi\mathbb Z) &\to &H^k(S^1,\R^3_\ast) \\
(u_1,u_2,u_3) &\mapsto& \left( u_1, u_2 + \int u_1^{-2}u_3, u_3\right)
\end{array} \right.\,,
\]
with $\int u_1^{-2}u_3$ denoting the indefinite integral of $u_1^{-2}u_3$. Then $\Ph$ is a 
bijection, when restricted to 
\[
\Ph : H^k(S^1,\R^3_\ast) \cap C_\al\i(2\pi\mathbb Z) \cap \{ u_2 = 0 \} \overset{\cong}\longrightarrow H^k(S^1,\R^3_\ast) 
\cap C_\al\i(2\pi\mathbb Z) \cap H_{\on{diff}}\i(0)\,.
\]
This implies that
\[
H^k(S^1,\R^3_\ast) \cap C_\al\i(2\pi\mathbb Z) \cap C_{\on{diff}}\i(0) \hookrightarrow H^k(S^1,\R^3_\ast) 
\cap C_\al\i(2\pi\mathbb Z)
\]
is a submanifold. Finally, that
\[
H^k(S^1,\R^3_\ast) \cap C_\al\i(2\pi\mathbb Z) \cap C_{\on{diff}}\i(0) 
\cap C_{\on{cl}}\i(0) \hookrightarrow H^k(S^1,\R^3_\ast) \cap C_\al\i(2\pi\mathbb Z) \cap C_{\on{diff}}\i(0)
\]
is a submanifold can again be shown via the implicit function theorem.
\end{proof}

\begin{lemma}
\label{lem:m3_proj_hk}
Let $k\geq 2$. The orthogonal projection $\on{Proj}^{\on{im}}$ defined in \eqref{eq:m3_orth_proj} 
can be extended to a smooth map 
\[
\on{Proj}^{\on{im}} : TH^k(S^1,\R^3_\ast) \to TH^k(S^1,\R^3_\ast)\,.
\]
\end{lemma}

\begin{proof}
First we show that $\on{Proj}^{\on{im}}_{\on{op}} : TH^k \to TH^k$ is a smooth map. Using the 
notation of Lem.~\ref{lem:m2_orth_proj_op} we see that for $(q,h) \in TH^k$ we have $A 
\in H^k(S^1,\R)$ and $B \in H^{k-1}(S^1,\R)$. The second component $k_2$ of 
$\on{Proj}^{\on{im}}_{\on{op}}$ is given as the solution of the equation   
\[
q_1^2k_2 - (Ak_2')' = B
\]
and by Lem.~\ref{lem:pde_regular} we know that such a $k_2 \in H^{k+1}(S^1,\R)$ exists and depends 
smoothly on $(q,h)$. Then \eqref{eq:m3_orth_proj_op} shows that $k_1,k_3 \in H^k(S^1,\R)$ thus 
showing $\on{Proj}^{\on{im}}_{\on{op}}$ to be well-defined and smooth.  

The smoothness of the maps $q \mapsto \on{grad}^{L^2} H_{\on{cl}}^i(q)$ for $i=1,2$ given by 
\eqref{eq:m3_hcl_grad} is clear and we see by inspection that $\on{grad}^{L^2} H_{\on{cl}}^i(q) 
\in H^k(S^1,\R^3)$. Therefore $q \mapsto \on{Proj}^{\on{im}}(q).\on{grad}^{L^2} H_{\on{cl}}^i(q) =: 
v^i(q)$ is also smooth. Let $w^1(q), w^2(q)$ be an orthonormal basis of $\on{span}\{v^1(q), 
v^2(q)\}$, constructed, e.g., via Gram-Schmidt. Then    
\[
T(q).h = h - \langle h, w^1(q) \rangle w^1(q) - \langle h, w^2(q) \rangle w^2(q)
\]
is smooth as well. Thus we conclude that the composition
\[
\on{Proj}^{\on{im}}(q).h = T(q).\on{Proj}^{\on{im}}_{\on{op}}(q).h
\] 
is smooth as required.
\end{proof}

\subsection{Well-posedness for the first and second metrics}
\label{Sec:wp_1_and_2}

The statements of Thm.~\ref{thm:m3_spray_smooth} and Cor.~\ref{cor:m3_solutions_smooth} also hold 
for the metrics \eqref{metric1} and \eqref{metric2} from Sect.~\ref{Sec:metric1} and Sect.~
\ref{Sec:metric2} on the space of closed curves. In the proof of Thm.~\ref{thm:m3_spray_smooth} we 
need to change the $R$-transform used to represent the metric and prove the analogues of 
Lem.~\ref{lem:sob_im_r} and Lem.~\ref{lem:m3_proj_hk}, the rest of the proof will remain the same.  

Let $G$ be the metric \eqref{metric2} from Sect.~\ref{Sec:metric2},
\[
G_c(h,k) = \int_{S^1} \langle D_s h, D_s k \rangle  + \langle D_s^2 h, n \rangle \langle D_s^2 k, n \rangle \ud s\,.
\]
The image of the $R$-transform \eqref{m2_rmap} is a submanifold in appropriate Sobolev extensions.

\begin{lemma}
\label{lem:m2_sob_im_r}
Let $k\geq 2$. The image of the $R$-transform
\[
R: \left\{
\begin{array}{ccc}
\on{Imm}^{k+1,k+2}(S^1,\R^2)/\on{Mot} &\to &H^k(S^1,\R_{>0}\x\R) \\
c &\mapsto &(\sqrt{|c'|}, \ka|c'|^2)
\end{array} \right.
\]
is an embedded submanifold of $H^k(S^1,\R_{>0}\x\R)$.
\end{lemma}

\begin{proof}
The image is given by
\[
\on{im}(R)=\{q\in H^k(S^1,\R_{>0}\times \R): H_{\on{cl}}(q)=0 \}\,,
\]
with the functional $H_{\on{cl}}$ given by
\[
H_{\on{cl}}(q)=\int_0^{2\pi} q_1^2 \exp\left( i\int_0^\th q_1(\si)^{-2}q_2(\si) \ud \si \right)\ud \th\,.
\]
The gradient of $H_{\on{cl}}$ was computed in Sect \ref{Sec:metric2:closed}. Since 0 is a regular 
point of $H_{\on{cl}}$, the statement of the lemma follows from the implicit function theorem in 
Banach spaces.  
\end{proof}

Since the image is defined by a finite number of constraints, the projection to the orthogonal 
complement can be written explicitely. 

\begin{lemma}
\label{lem:m2_proj_hk}
Let $k\geq 2$. The orthogonal projection $\on{Proj}^{\on{im}}$ to $T\on{im}(R)$ can be extended to 
a smooth map 
\[
\on{Proj}^{\on{im}} : TH^k(S^1,\R_{>0}\x\R) \to TH^k(S^1,\R_{>0}\x\R)\,.
\]
\end{lemma}

\begin{proof}
The smoothness of the maps $q \mapsto \on{grad}^{L^2} H_{\on{cl}}^i(q)$ for $i=1,2$ given by 
\eqref{eq:m2_grad_hcl} is clear and we see by inspection that $v^i(q) := \on{grad}^{L^2} 
H_{\on{cl}}^i(q) \in H^k(S^1,\R^2)$. Let $w^1(q), w^2(q)$ be an orthonormal basis of 
$\on{span}\{v^1(q), v^2(q)\}$, constructed, e.g., via Gram-Schmidt. Then the orthogonal projection 
is given by    
\[
\on{Proj}^{\on{im}}(q).h = h - \langle h, w^1(q) \rangle w^1(q) - \langle h, w^2(q) \rangle w^2(q)
\]
and is smooth.
\end{proof}

For the metric \eqref{metric1},
\[
G_c(h,k)=\int_{S^1} \ka^{-3/2}\langle D^2_sh,n\rangle\langle D^2_sk,n\rangle + \langle D_sh,v\rangle\langle D_sk,v\rangle\ud s\,,
\]
the analogues of the above lemmas can be proven in the same way.

\section{Discretization}\label{sec:discretization}

In this section we describe, how one can use the $R$-transform to discretize the geodesic equation of second order metrics. To make the exposition more concise we will restrict ourselves to the third metric, described in Sect. \ref{Sec:metric3}, although the principles are rather general.

We consider the metric \eqref{metric3},
\[
G_c(h,k) = \int_M \langle D_s h, D_s k \rangle  + \langle D_s^2 h, n \rangle \langle D_s^2 k, n \rangle \ud s\,.
\]
The space $\on{Imm}/\on{Tra}$ with the metric $G$ is isometric to
\[
\on{im}(R) = \left\{ (q_1,q_2,q_3) \,:\,
\begin{array}{l}
(1)\, q_2' = q_1^{-2}q_3 \\
(2)\, \int_{S^1} q_1^2 \exp(i q_2) \ud \th=0
\end{array} \right\} \,,
\]
which is a submanifold of $C^\infty(S^1, \R^3_\ast)$ equipped with the $G^{L^2}$-metric. Here $\R^3_\ast = \R_{>0}\x S^1\x \R$ and we equip it with the non-flat Riemannian metric
\[
g = 4\, dq_1\otimes dq_1+ q_1^2\, dq_2\otimes dq_2 + q_1^{-6}\, dq_3\otimes dq_3\,.
\]
Instead of discretizing the space $\on{Imm}/\on{Tra}$ and the geodesic equation thereon, we discretize instead $C^\infty(S^1,\R^3_\ast)$, the metric $G^{L^2}$ and the constraints defining $\on{im}(R)$.

\subsection{Spatial discretization}

We replace the curve $q \in C^\infty(S^1,\R^3_\ast)$ by $N$ uniformly sampled points $q^1,\dots,q^N$ with $q^k = q(2\pi k/N)$. Denote by $\De \th = 2\pi/N$ the spatial resolution. The continuous geodesic equation corresponds to a Hamiltonian system with the Hamiltonian
\[
E_{\on{cont}}(q, p) = \frac 12 \int_{S^1} g\i_{q(\th)}(p(\th), p(\th)) \ud \th\,,
\]
together with the constraint functions $H_{\on{diff}}$, $H_{\on{cl}}$ defined in \eqref{eq:m3_constraints}, that define the image of the $R$-transform. Instead of discretizing the geodesic equations directly, we discretize the Hamiltonian function. The discrete Hamiltonian is
\[
E_{\on{discr}}(q^1,\dots,q^N,p^1,\dots,p^N) = 
\frac 12 \sum_{k=1}^N g\i_{q^k}(p^k, p^k) \De \th\,.
\]
To simplify notation we shall denote the discretized curve $(q^1,\dots,q^N)$ again by $q$ and the same for the momentum. The discrete constraint functions are
\begin{align*}
H^k_{\on{diff}}(q,p) &= \left(q_1^k\right)^{-2} q_3^k - \left( q_2^{k+1} - q_2^k\right) / \De \th \\
H_{\on{cl}}(q,p) &= \sum_{k=1}^N \on{exp}(iq_2^k) (q_1^k)^2 \De \th\,.
\end{align*}
Thus we have replaced an infinite-dimensional system by a $3N$-dimensional Hamiltonian system with $N+2$ constraints. The resulting system has $2N-2$ degrees of freedom. This corresponds to discretizing a plane curve by $N$ points and removing translations, again leading to $2N-2$ degrees of freedom.

The advantage of the $R$-transform is that the Hamiltonian $E_{\on{cont}}$ of the continuous system doesn't contain spatial derivatives, since it is related to an $L^2$-type metric. The spatial derivatives appear in the constraints, in particular in $H_{\on{diff}}$, which enforces that the third component $q_3$ is a derivative of the second component $q_2$. However, even though $q_3$ represents the curvature of the curve and thus a second derivative, the constraint $H_{\on{diff}}$ is written in terms of the first derivative only, which can be discretized using first order differences.

Our discrete equations are now Hamilton's equations of a constrained Hamiltonian system. Denote by $H: \R^{3N} \to \R^{N+2}$ the collected constraint functions and let $\la \in \R^{N+2}$ be a Lagrange multiplier. Then Hamilton's equations are
\begin{align}
\label{eq:findim_hamilton}
\begin{split}
\p_t q &= \p_{p} E_{\on{discr}}(q, p) \\
\p_t p &= -\p_{q} E_{\on{discr}}(q, p) + DH^T(q).\la \\
H(q) &= 0\,.
\end{split}
\end{align}

\subsection{Time discretization}

There is a variety of integrators available for the time-discretization of a constrained Hamiltonian system. For example RATTLE is a second-order symplectic method, that preserves constraints exactly. A time-step of RATTLE is given by the following equations. To simplify notation, we denote in this section by $q^j=q(t_j)$ the system at time $t_j$.
\begin{equation}
\label{eq_rattle}
\begin{split}
p^{j+1/2} &= p^j + \frac {\De t}{2} \left( -\p_q E(q^j, p^{j+1/2}) + DH^T(q^j).\la_1 \right) \\
q^{j+1} &= q^j + \frac {\De t}{2} \left( \p_p E(q^j, p^{j+1/2}) + \p_p E(q^{j+1}, p^{j+1/2})\right) \\
0&= H(q^{j+1})\\
p^{j+1} &= p^{j+1/2} + \frac{\De t}{2} \left( -\p_q E(q^{j+1}, p^{j+1/2}) + DH(q^{j+1})^T.\la_2 \right) \\
0 &= DH(q^{j+1}).\p_p E(q^{j+1},p^{j+1})\,.
\end{split}
\end{equation}
This method was first proposed in \cite{Jay1996}. One first performs a momentum update with half of the timestep using the implicit Euler method and an unknown Lagrange multiplier $\la_1$. This is followed by a full time-step for the position using the implicit midpoint rule. The Lagrange multiplier $\la_1$ is determined by the condition $H(q^{j+1})$, which guarantees that the constraints are exactly satisfied in each time-step. Then we perform another half time-step for the momentum with the explicit Euler and determine the Lagrange multiplier $\la_2$ by requiring the hidden constraint $DH(q^{j+1}).\p_p E(q^{j+1},p^{j+1}) = 0$ to be satisfied. See \cite{Hairer2006, Leimkuhler2004} for more details about symplectic integrators.

\section{Experiments}
In this section we present a series of numerical examples to demonstrate the value of $R$-transforms for numerical computations. The examples were computed as described in Sect.~\ref{sec:discretization}.
In all these examples we will only consider the third metric, i.e.,
$$G_c(h,h)=\int_{S^1}\langle D_sh, D_sh \rangle+\langle D_s^2h, n \rangle^2 ds\;.$$
The curves are discretized with $100$ points and since the metric ignores translations we centered all curves such that their center of mass lies at the origin. 
In Fig. \ref{fig1} we show two examples of solutions to the geodesic boundary value problem.
\begin{figure}[ht]
\includegraphics[width=0.34\textwidth]{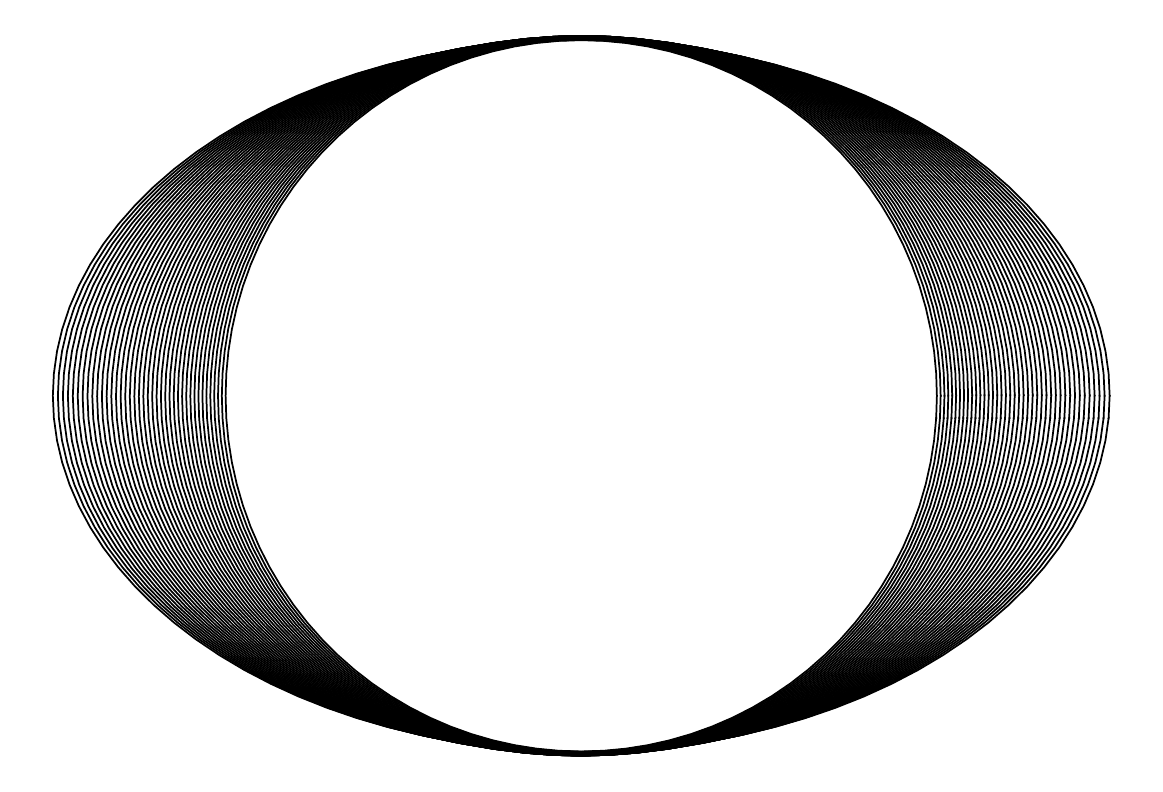}
\includegraphics[width=0.65\textwidth]{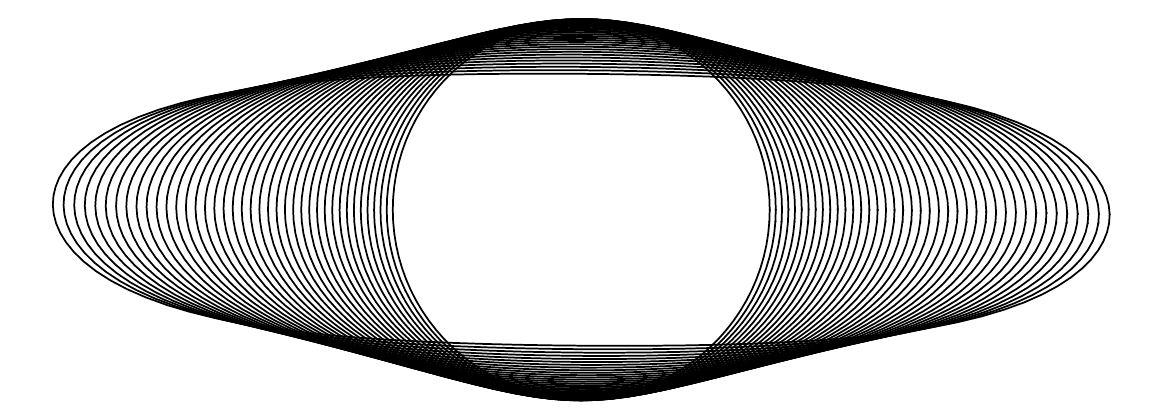}\\
\includegraphics[width=0.34\textwidth]{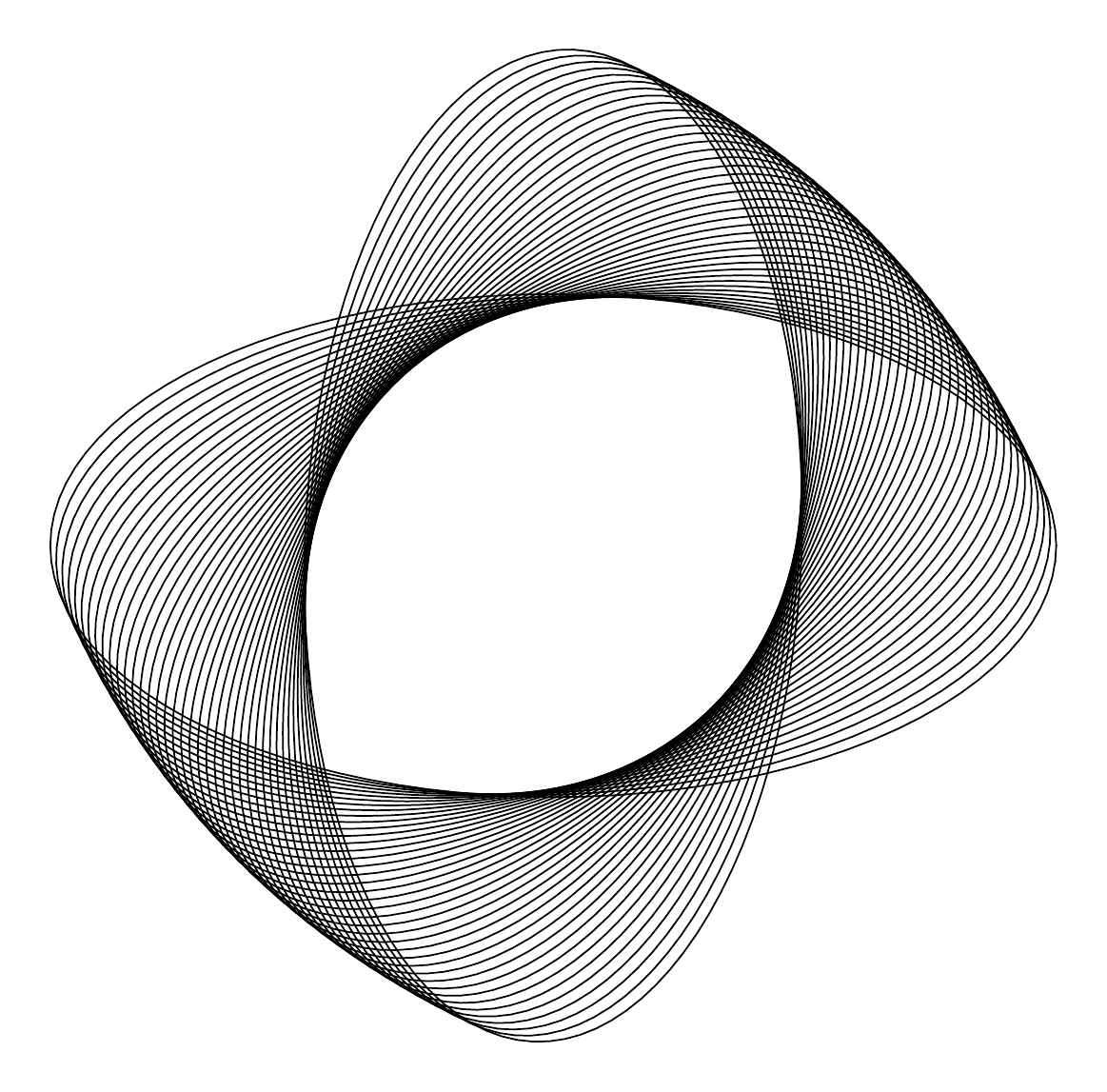}
\includegraphics[width=0.65\textwidth]{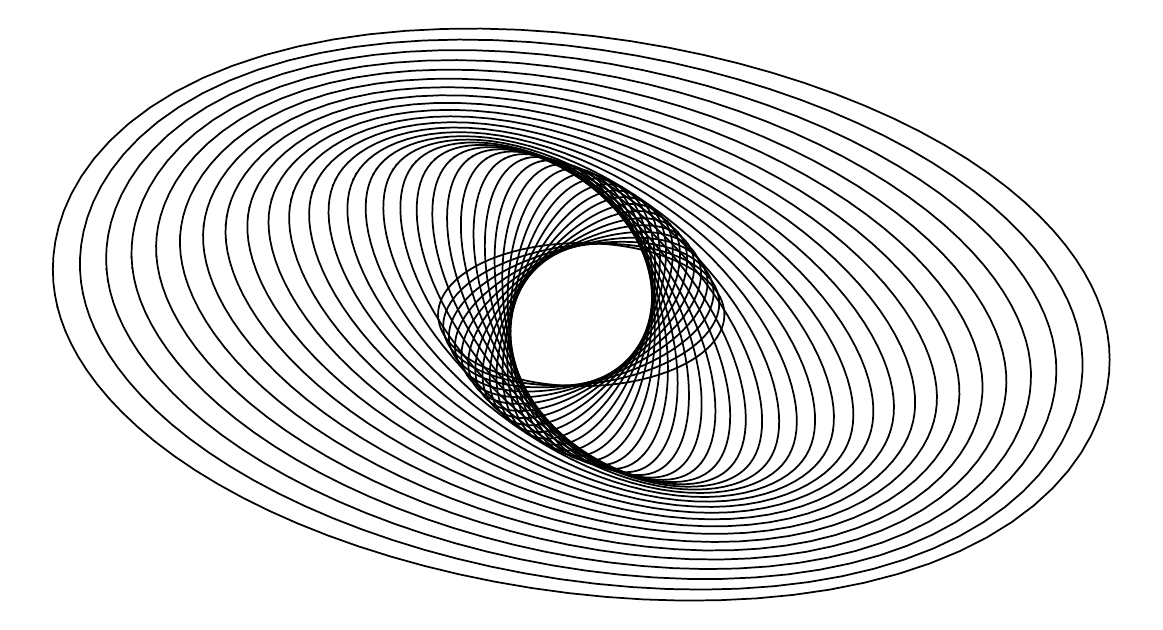}
\caption{First line, left side: Geodesic connecting a circle to an ellipse in time $t=2$. First line, right side: The geodesic continued until time $t=6$.
Second line, left side: Geodesic connecting a ellipse to a rotated ellipse in time $t=2$. Second line, right side: The geodesic continued until time $t=6$.}
\label{fig1}
\end{figure}

The second series of examples (Fig. \ref{fig2}) is concerned with the geodesic initial value problem. It shows two geodesic that starts at the circle with two different initial velocities. 
\begin{figure}[ht]
\includegraphics[width=0.40\textwidth]{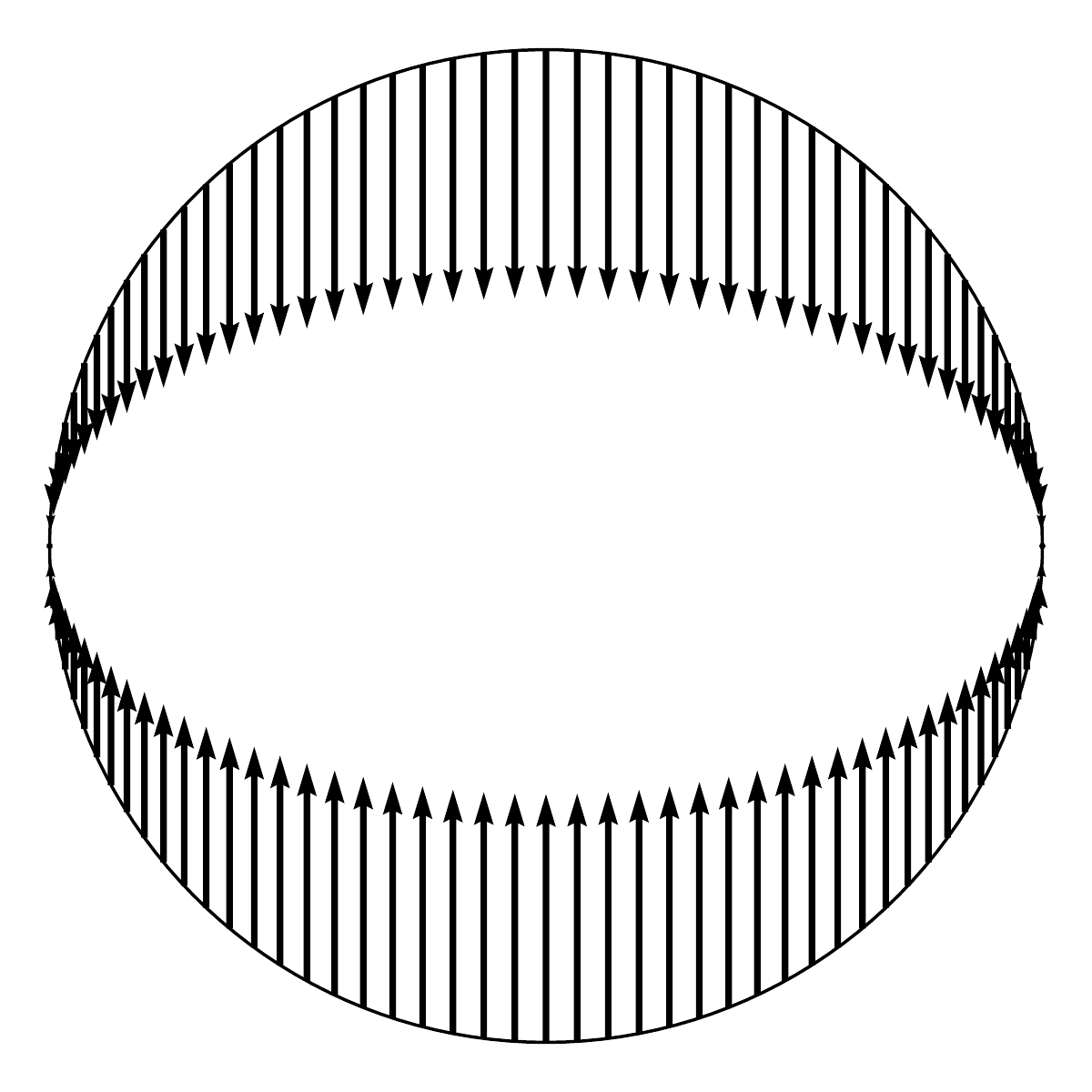}\hspace{.5cm}
\includegraphics[width=0.40\textwidth]{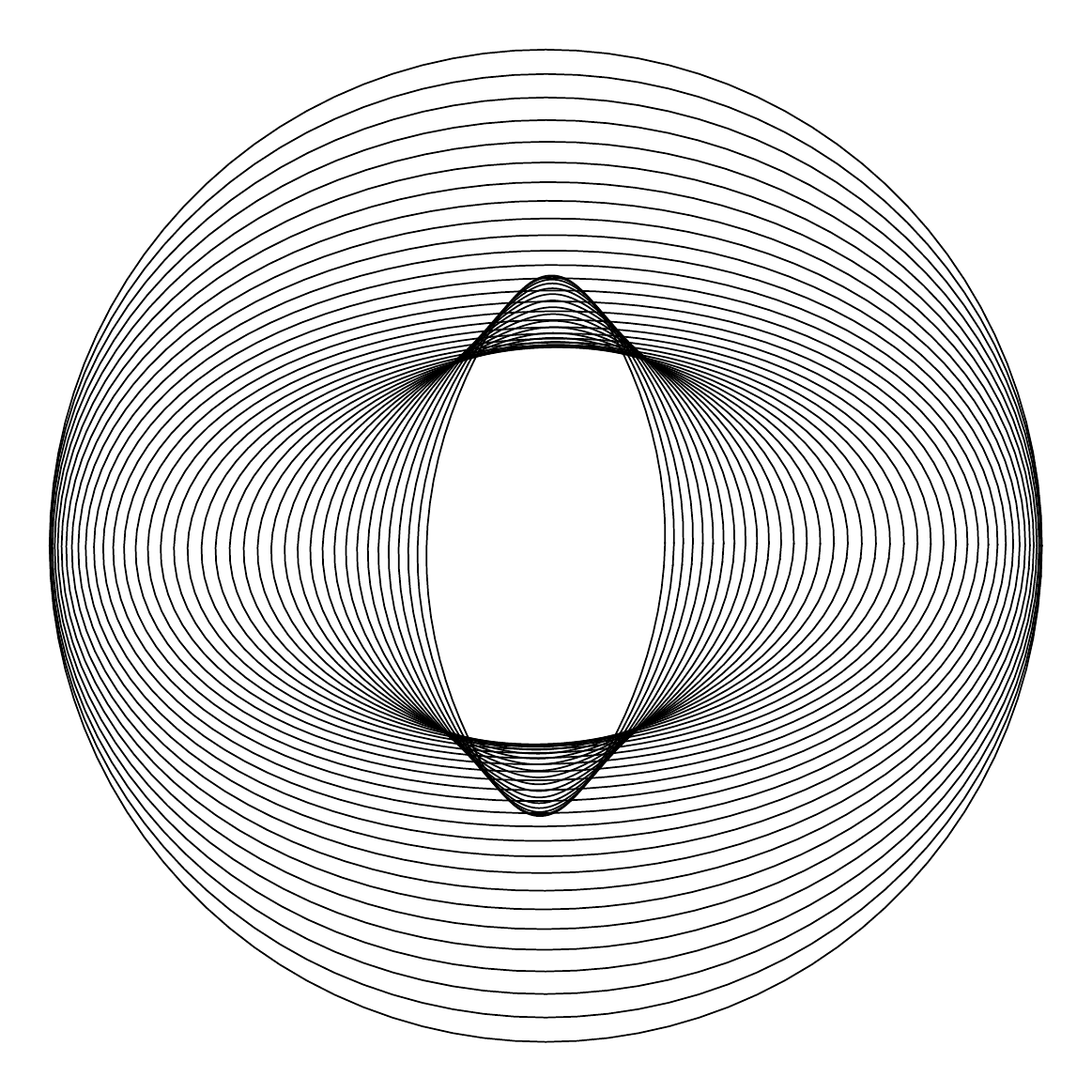}
\includegraphics[width=0.40\textwidth]{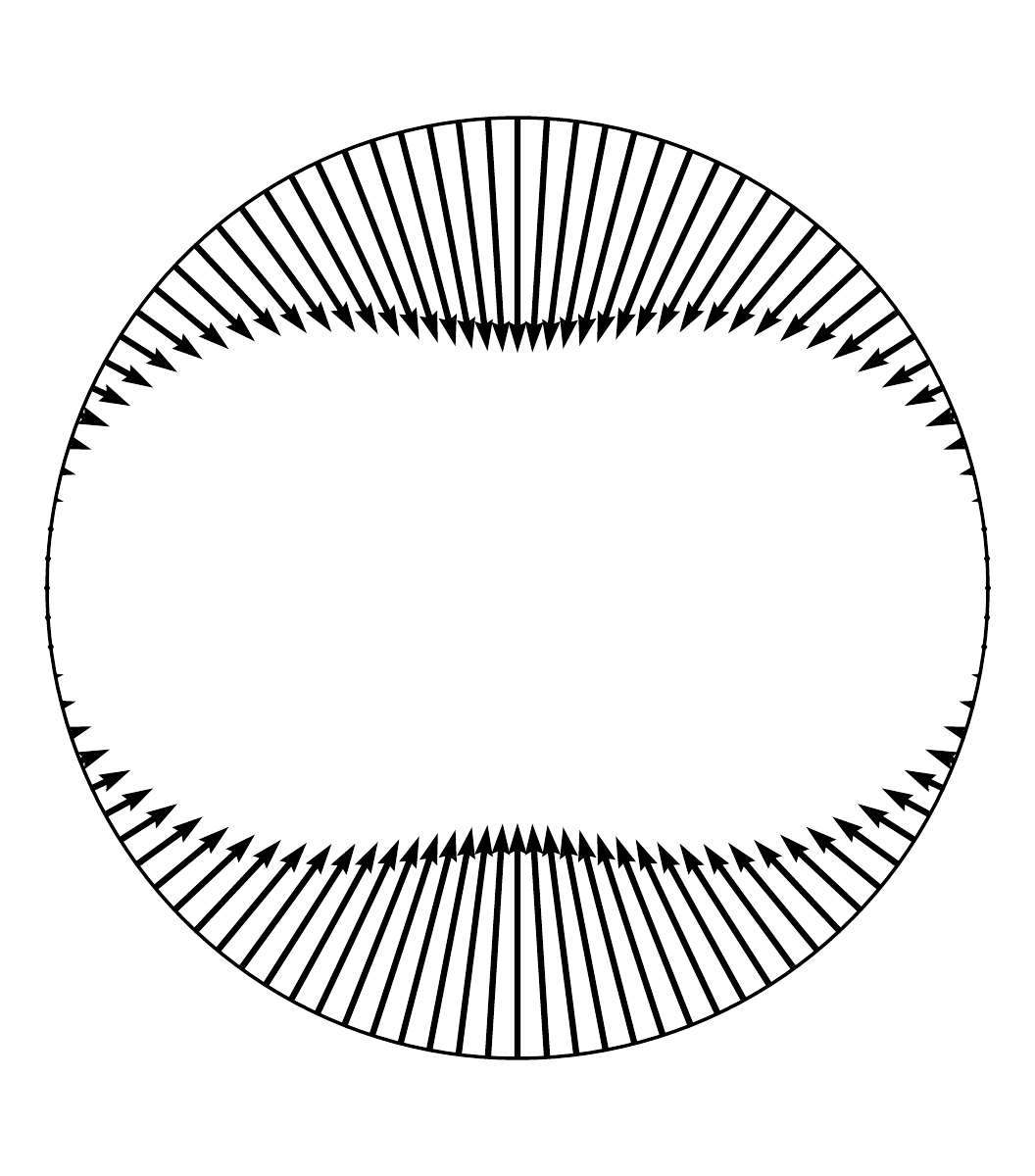}\hspace{.5cm}
\includegraphics[width=0.40\textwidth]{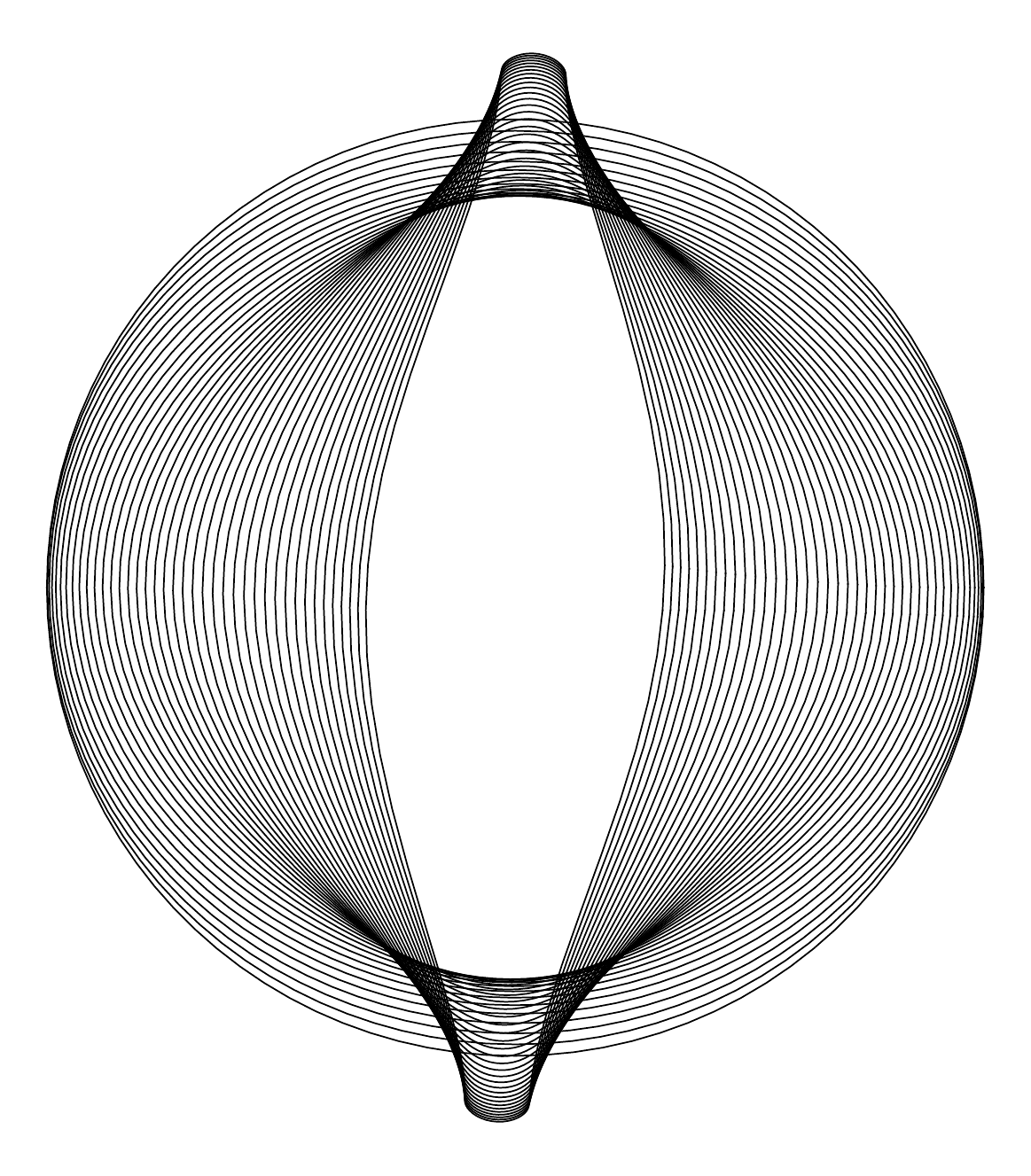}
\caption{Two examples of a geodesic that starts at the circle. On the right hand side the geodesic is printed, whereas on the left hand side we pictured the initial velocity. 
First line: Initial velocity $h=(0,\on{sin}\th)$. The geodesic is computed  until time $2$.
Second line: Initial velocity $h=-\on{sin}^2\th(\on{cos}\th,\on{sin}\th)$. The geodesic is computed  until time $1$. }
\label{fig2}
\end{figure}

The last example shows a geodesic in the shape space of umparametrized curves.
Following the presentation of \cite{Michor2007} we  identify this space with the quotient space $$B_i(S^1,\R^2)=\on{Imm}(S^1,\R^2)/\on{Diff}(S^1).$$
Every reparametrization invariant metric on  $\on{Imm}(S^1,\R^2)$ induces  a metric on the shape space $B_i(S^1,\R^2)$ such that the projection  
$$
\pi: \on{Imm}(S^1,\R^2) \to B_i(S^1,\R^2)
$$
is a Riemannian submersion. In this setting geodesics on shape space $B_i(S^1,\R^2)$ correspond to horizontal geodesics on $\on{Imm}(S^1,\R^2)$, i.e., geodesics on $\on{Imm}(S^1,\R^2)$
with horizontal velocity. By the conservation of reparametrization momentum a geodesic with horizontal inital velocity stays horizontal for all time and thus this condition has to be checked at the initial point of the geodesic only.
For a detailed description of this construction see \cite{Michor2007,Bauer2011b}.

\begin{figure}[ht]\label{geod_shapespace}
\includegraphics[width=0.49\textwidth]{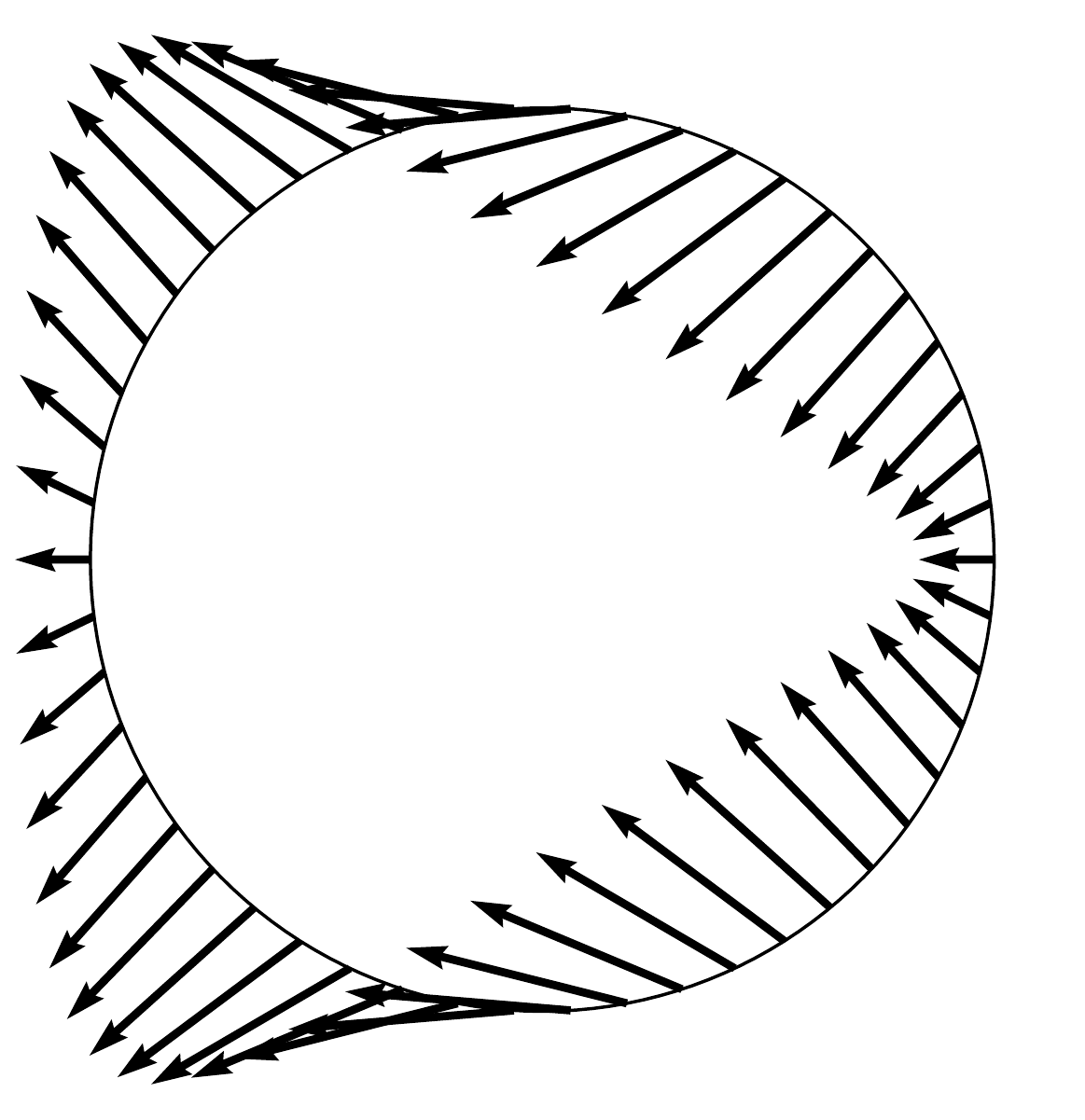}
\includegraphics[width=0.49\textwidth]{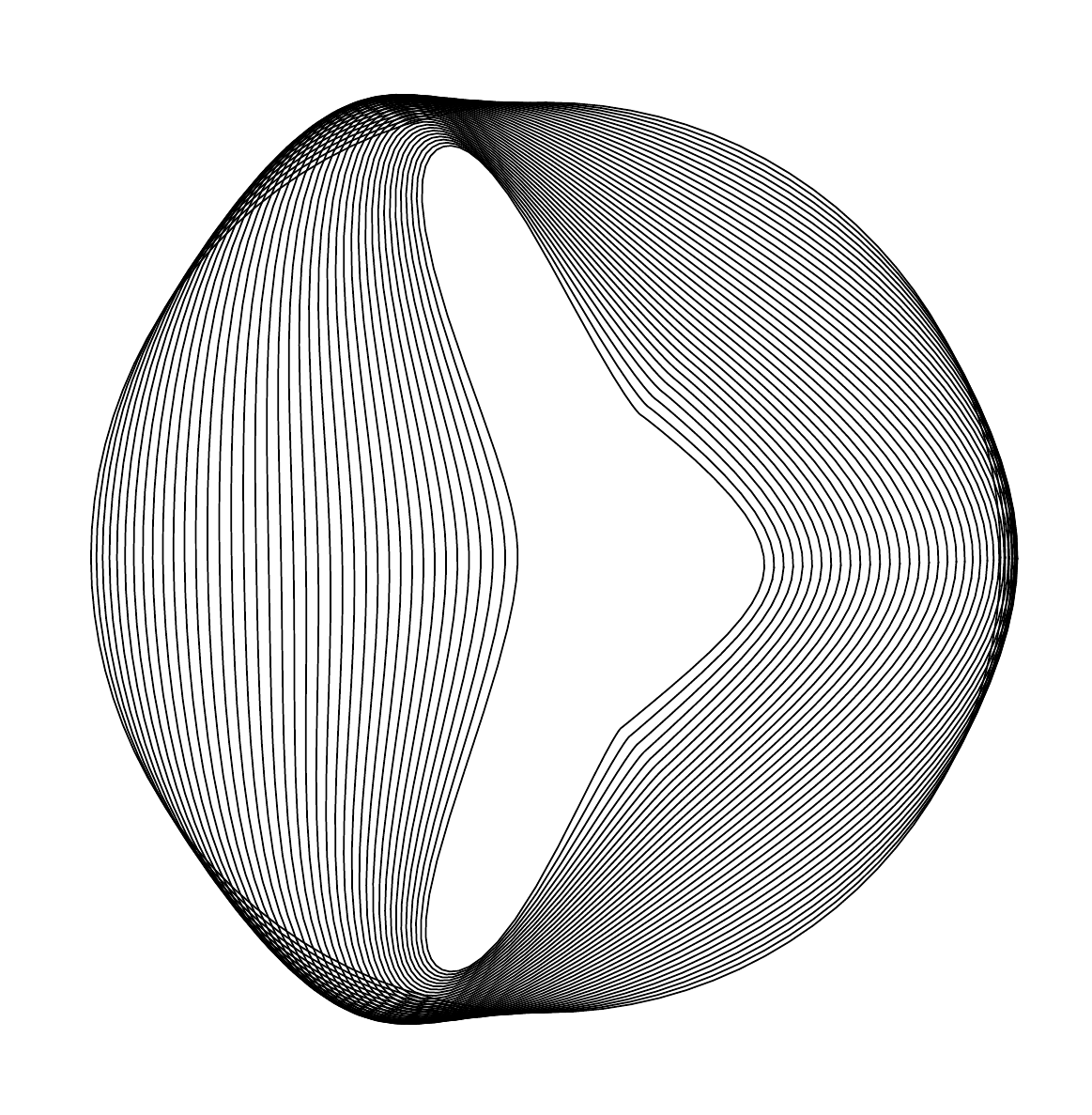}
\caption{A horizontal geodesic on $\on{Imm}(S^1,\R^2)$ with initial velocity $h=-(2-\on{cos}2\th,2\on{sin}2\th)$ until time $.3$.}
\end{figure}
For metrics that are induced by a differential operator field $L$ the horizontality condition can be expressed as
$$h\in \on{Hor}(c)\subset T_c\on{Imm}(S^1,\R^2)\Leftrightarrow L_ch= f.n,\quad f\in C^\infty(S^1)\,,$$
with $n$ denoting the normal field to $c$.
In the following we want to investigate this condition for the third metric,
\[
G_c(h,h)=\int_{S^1}\langle D_sh, D_sh \rangle+\langle D_s^2h, n \rangle^2 ds\,.
\]
This metric is induced by the differential operator
\[
L_c h = D_s^2 \big(\langle D_s^2 h, n \rangle n \big)  -D^2_sh\,.
\]
To simplify the expressions we choose the circle $c(\th)=(\on{cos}\th,\on{sin}\th)$ as the starting point of the geodesic.
Then we have
$$|c'|=1,\quad D_s=\partial_{\theta},\quad \partial_\theta n=v\quad\text{and}\quad\partial_\theta v=-n\,.$$ 
Let $h=an+bv$. Then
\begin{align*}
 D_s^2 h&=\partial^2_{\theta}(an+bv)=\partial_{\theta}(a'n+an'+b'v+bv')\\
&=a''n+2a'n'+an''+b''v+2b'v'+bv''\\&=(a''-a+2b')n+(-2a'+b''-b)v
\end{align*}
Thus we have
\begin{align*}
\partial^2_{\theta} \big(\langle \partial^2_{\theta} h, n \rangle n \big)&=\partial^2_{\theta}\big((a''-a-2b')n\big)\\
&=( a'''' - a'' + 2b''')n + 2(a'''-a'+2b'')n'
+(a''-a+2b')n'' \\
&=-2(a'''-a'+2b'')v+(a''''-2a''+a-2b'+2b''')n
\end{align*}
From this we can read off the tangential and the normal part of $Lh$:
\begin{align*}
 \langle L_c h,v\rangle&= -2a'''+4a'-5b''+b\\
  \langle L_c h,n\rangle&=a''''-3a''+2a+2b'''-4b'
\end{align*}
A horizontal velocity satisfies $\langle L_c h,v\rangle=0$, which leads to the ODE
$$-2a'''+4a'-5b''+b=0\,.$$
A solution to this equation is given for example by 
$$a=3 \on{cos}\th,\quad b=\on{sin}\th\,,$$
This leads the initial velocity
$$h=(-\on{cos}^2 \th-3\on{sin}^2 \th,-4\on{cos}\th \on{sin}\th )=-(2-\on{cos}2\th,2\on{sin}2\th)\,.$$

\appendix
\counterwithin{theorem}{section}

\section{Translation invariant sprays}
\label{Sec:trans_inv}

Let $N$ be a finite dimensional manifold. Consider a spray $\Xi$ on $C^\infty(S^1, N)$ that has 
smooth extensions to $H^k(S^1,N)$ for $k\geq k_0$. For each initial condition $(q_0, v_0) 
\in TH^k(S^1,N)$ and each Sobolev order $k$ denote by $J_k(q_0, v_0)$ the maximal domain of 
existence of its flow. If $(q_0, v_0) \in TH^{k+1}(S^1,N)$ then a-priori we only have the inclusion 
\[
J_{k+1}(q_0, v_0) \subseteq J_k(q_0, v_0)\,.
\]
To obtain that smooth initial conditions have smooth solutions need to know that the length of the 
maximal existence interval is bounded from below, i.e., $\cap_{k\geq k_0} J_k(q_0, v_0) = \{0\}$ 
cannot happen. If the spray is invariant under translations, this is ruled out by the following 
result.   

\begin{theorem}[Ebin-Marsden, 1970]
\label{thm:max_domain}
Let the spray $\Xi$ be invariant under translations, i.e.
\[
\Xi(T(\si)q, T(\si)v) = T(\si)\Xi(q,v)\,,
\]
with $T(\si)q(\th) = q(\th + \si)$ being the translation group. Then for initial conditions $(q_0, 
v_0) \in TH^{k+1}(S^1,N)$ we have 
\[
J_{k+1}(q_0, v_0) = J_{k}(q_0, v_0)\,.
\]
\end{theorem}

\begin{proof}
This result is implicit in \cite[Thm.~12.1]{Ebin1970} and made explicit in \cite[Lem.~
5.1]{Escher2012pre}. We prove it here only for $N=\R^d$, the general case requiring only slightly 
more cumbersome notation.  

If the spray is translation-invariant, then so is the exponential map,
\[
\exp(T(\si)q_0, T(\si)tv_0) = T(\si)\exp(q_0, tv_0)\,.
\]
Differentiating at $\si=0$ gives
\[
T_{(q_0, tv_0)} \exp.(q_0', tv_0') = \p_\th \exp(q_0, tv_0)\,.
\]
From our assumption $q_0, v_0 \in H^{k+1}(S^1,\R^d)$ it follows that the left-hand side is an 
element of $H^k(S^1,\R^d)$. Since $q(t) = \exp(q_0, tv_0)$ we see that $q'(t) \in H^k(S^1,\R^d)$ 
which implies $q(t) \in H^{k+1}(S^1,\R^d)$. Thus $J_k(q_0, v_0) = J_{k+1}(q_0, v_0)$.  
\end{proof}

\section{Regularity for elliptic equations}

\begin{lemma}
\label{lem:pde_regular}
Let $k\geq 2$ and let $L$ be the operator 
\[
Lu = -(au')' + bu
\] 
with $a \in H^{k-1}(S^1)$, $b \in H^{k-2}(S^1)$ and $a>0$, $b\geq\ep>0$ for $\ep \in \R$. The $L$ 
is a bibounded, invertible operator 
\[
L: H^k(S^1) \to H^{k-2}(S^1)\,.
\]
Furthermore the map $L\i: a, b, f \mapsto L_{a,b}\i f$ is a smooth map
\[
L\i : H^{k-1}(S^1) \x H^{k-2}(S^1) \x H^{k-2}(S^1) \to H^s(S^1)\,.
\]
\end{lemma}

\begin{proof}
This lemma can be proven in the same way as existence and regularity results are proven for second order elliptic PDEs in, e.g., \cite[Chap. 6]{Evans2010}. The proofs can be followed line by line, even though we have required less regularity for the coefficient functions.
\end{proof}

%\bibliographystyle{abbrv}
%\bibliography{h2-bib,../../ref/ref,../../ref/preprints,../../ref/articles}
%\bibliography{h2-bib}%articles.bib,preprints.bib}

\end{document}